\documentclass{amsart}%
\usepackage{amsmath, amssymb, amsfonts}
\usepackage{graphicx}
\usepackage{amsmath}
\usepackage{amsfonts}
\usepackage{amssymb}%
\providecommand{\U}[1]{\protect \rule{.1in}{.1in}}
\newtheorem{theorem}{Theorem}[section]

\newtheorem{axiom}[theorem]{Axiom}

\newtheorem{corollary}[theorem]{Corollary}

\newtheorem{definition}[theorem]{Definition}

\newtheorem{lemma}[theorem]{Lemma}
\newtheorem{notation}[theorem]{Notation}

\newtheorem{proposition}[theorem]{Proposition}
\newtheorem{remark}[theorem]{Remark}

\newcommand*{\st}
{\mathrm{st}}
\newcommand*{\forallst}
{\forall^\mathrm{st}}

\begin{document}
\title{Axiomatics for the external numbers of nonstandard analysis}
\author{Bruno Dinis}
\address[B. Dinis]{Departamento de Matem\'{a}tica, Faculdade de Ci\^{e}ncias da
Universidade de Lisboa, Campo Grande, Ed. C6, 1749-016, Lisboa, Portugal}
\email{bmdinis@fc.ul.pt}
\author{Imme van den Berg}
\address[I.P. van den Berg]{Departamento de Matem\'{a}tica, Universidade de \'{E}vora, Portugal}
\email{ivdb@uevora.pt}
\thanks{The first author acknowledges the support of the Funda\c{c}\~{a}o para a
Ci\^{e}ncia e a Tecnologia, Portugal [grant SFRH/BPD/97436/2013]}

\begin{abstract}
Neutrices are additive subgroups of a nonstandard model of the real numbers.
An external number is the algebraic sum of a nonstandard real number and a
neutrix. Due to the stability by some shifts, external numbers may be seen as
mathematical models for orders of magnitude. The algebraic properties of
external numbers gave rise to the so-called solids, which are extensions of
ordered fields, having a restricted distributivity law. However, necessary and
sufficient conditions can be given for distributivity to hold. In this article
we develop an axiomatics for the external numbers. The axioms are similar to,
but mostly somewhat weaker than the axioms for the real numbers and deal with
algebraic rules, Dedekind completeness and the Archimedean property. A
structure satisfying these axioms is called a complete arithmetical solid. We
show that the external numbers form a complete arithmetical solid, implying
the consistency of the axioms presented. We also show that the set of precise
elements (elements with minimal magnitude) has a built-in nonstandard model of
the rationals. Indeed the set of precise elements is situated between the
nonstandard rationals and the nonstandard reals whereas the set of non-precise
numbers is completely determined.

\end{abstract}
\maketitle

\section{Introduction}

Consider a nonstandard model of the real number system $^{\ast}\mathbb{R}$. A
\emph{neutrix} is an additive convex subgroup of $^{\ast}\mathbb{R}$ and an
\emph{external number} is the algebraic sum of a nonstandard real number with
a neutrix. In such a nonstandard framework there are many neutrices such as
$\oslash$, the external set of all infinitesimals, and $\pounds $, the
external set of all limited numbers, i.e. numbers bounded in absolute value by
a standard number.

Typically external numbers have neither infimum nor supremum, being stable for
some translations, additions and multiplications. As argued in
\cite{koudjetivandenberg} and \cite{dinisberg 2011}, they are models of orders
of magnitude or transitions with imprecise boundaries of Sorites type
\cite{dinis 2017}. With external numbers it is possible to work directly with
imprecisions and errors without recourse to upper bounds. They generate a
calculus of propagation of errors not unlike the calculus of real numbers,
allowing for total order and even for a sort of generalized Dedekind
completeness property. Some applications in asymptotics, singular
perturbations, linear algebra and statistics are contained in
\cite{koudjetivandenberg}, \cite{Justino} and \cite{neutrixdecompositie}, and
the references mentioned in the latter article. The term neutrix is borrowed
from Van der Corput \cite{Van der Corput}. His neutrices are rings of
"neglectable functions". The calculation rules satisfied by the external
numbers are significantly stronger than the functional asymptotic calculus of
$o$'s and $O$'s \cite{debruijn} and Van der Corput's neutrices which for
instance do not respect total order.

Algebraic properties of external numbers have been studied in
\cite{koudjetithese}, \cite{koudjetivandenberg} and \cite{dinisberg 2011}.
Respecting total order, they are based on semigroup operations more than group
operations, a sort of "mellowed" version of the common rules of calculation of
real numbers. In \cite{dinisberg 2015 -1} it was shown that the set of cosets
of non-Archimedean ordered fields with respect to all possible convex
subgroups for addition has a similar algebraic structure. Such structures were
called \emph{solids}. Elements of a solid are the sum of a \emph{precise}
element and a \emph{magnitude}. The magnitudes act as individualized neutral
elements and correspond to the convex subgroups, and the precise elements,
i.e. elements with magnitude "zero", to the elements of the underlying ordered
field. Solids are not completely distributive, but necessary and sufficient
conditions can be given for triples of elements to satisfy distributivity
\cite{dinisberg 2015 -2}.

In this article we extend the axiomatic laws of solids. The axioms added
include multiplicative properties of neutrices, a generalized Dedekind
completeness property, and an Archimedean property. The multiplicative axioms
are inspired by the results of \cite{koudjetithese} and
\cite{koudjetivandenberg} on neutrices which are idempotent for
multiplication. The axioms determine the product of idempotent neutrices, in
fact of all neutrices, because it is postulated that every neutrix is a
multiple of an idempotent neutrix. In $ZFC$ the structure of real numbers
$\mathbb{R}$ is characterized up to isomorphism in a second-order language, as
the unique Dedekind complete ordered field, or equivalently as the unique
Archimedean complete ordered field in which Cauchy sequences converge. However
second-order properties of nonstandard models of the reals are less obvious.
Therefore we intend to remain in a first-order language and so the generalized
Dedekind completeness axiom is stated in the form of a scheme. In order to
deal with the Archimedean property we assume some Peano-like axioms, including
a scheme on induction. Due to this Archimedean property models must include a
copy of the nonstandard integers, hence of the nonstandard rationals, and due
to the generalized Dedekind completeness it must be possible to embed models
in the nonstandard reals.

We prove consistency of the axiomatics in the setting of a nonstandard model
$^{\ast}\mathbb{R}$, which has the form of an adequate ultralimit for a
bounded version of Nelson's syntactical Reduction Algorithm \cite{Nelsonist}
to hold. Admissible models will be called \emph{complete arithmetical solids}.
Up to isomorphism, once the magnitudes are specified the set of non-precise
numbers of a complete arithmetical solid is completely determined as sums of a
nonstandard rational and a magnitude. For the set of precise elements of a
model we give upper and lower bounds, in fact the precise elements are
situated between the nonstandard rationals and the nonstandard reals. As a
result, complete arithmetical solids come closer to a syntactical
characterization of the external numbers than the solids of \cite{dinisberg
2015 -2} which can also be built on non-Archimedean ordered fields.

There have been various attempts to deal with the external algebraic and order
structure of the real line. Wattenberg \cite{Wattenberg} and Gonshor
\cite{Gonshor} developed a calculus of the lower halflines of the nonstandard
real line. However this gave rise to a less rich algebraic structure, and also
they do not consider completeness properties of Dedekind kind. Keisler and
Schmerl \cite{KeislerSchmerl} consider two other completeness properties of
the external real line in a model-theoretic setting, Scott completeness and
Bolzano-Weierstrass completeness, without developing an algebraic calculus.
From an axiomatic point of view Scott completeness and Bolzano-Weierstrass
completeness were reconsidered in \cite{kanoveireeken0}.

For sake of clarity and reference we start in Section \ref{Section Axioms}
with an overview of all the axioms.

The axiom scheme on Generalized Dedekind completeness is discussed in Section
\ref{Section Dedekind completeness}. It is stated in terms of precise numbers,
but typically concerns halflines which are not precise, i.e., stable under
some shifts. For precise definable halflines Generalized Dedekind completeness
reduces to ordinary Dedekind completeness, i.e., halflines including the
extremum, and halflines with the extremum just beyond. When applied to
halflines of a solid (thus including non-precise numbers), three types of
halflines occur, instead of two; it is shown that they are mutually exclusive.

With Generalized Dedekind completeness one can define the minimal magnitude
including unity, denoted by $\pounds $, and the maximal magnitude without
unity, denoted by $\oslash$. Section
\ref{Section Limited numbers infinitesimals} shows that their algebraic
properties correspond to a large extent to the properties of the limited
numbers, respectively the infinitesimals in a nonstandard model of the reals.
However they are somewhat weaker, and it is not possible to decide whether
$\oslash \pounds =\oslash$ or $\oslash \pounds =\pounds $. In Section
\ref{Section Product magnitudes} we choose the product of $\oslash$ and
$\pounds $ to be $\oslash$, in accordance with the fact that the product of an
infinitesimal and limited real is infinitesimal. The axiom is stated in terms
of the product of an idempotent magnitude with unity and its maximal ideal.
Ideals are defined by analogy to ideals of rings, and the existence of maximal
ideals follows from Generalized Dedekind Completeness. An axiom that says that
an arbitrary magnitude is a multiple of an idempotent magnitude enables to
settle the product of any two neutrices.

In Section \ref{section on consistency} we show relative consistency of our
axiomatics with $ZFC$ by the construction of a model based on external subsets
of an appropriate nonstandard model of the reals.

In Section \ref{Section Characterization properties} we show that, up to
isomorphism, the precise elements of a model are situated between the
nonstandard rationals and the nonstandard reals. In a sense, it is also
possible to identify a notion of standard part, here called shadow. It is
shown that the shadows of the precise elements are situated between the
shadows of the nonstandard rationals and the shadows of the nonstandard reals.
Finally we investigate the relation between the standard structure and the
nonstandard structure of a complete arithmetical solid. More precisely, we
show that the Leibniz rules hold for the precise elements of a complete
arithmetical solid, i.e. the precise elements are a model of the axiomatics
$ZFL$ \cite{Lutz}, \cite{Callot}, and that the "natural numbers" in a solid
are a model for the axiomatics given by Nelson in \cite{REPT} here called
$REPT$, but with external induction restricted to the language $\left \{
+,\cdot \right \}  $. These ``weak'' nonstandard axiomatics are not without
interest. Indeed, in \cite{Callot} Callot showed that a substantial part of
ordinary analysis can be carried out in $ZFL$ and in \cite{REPT} Nelson argued
that $REPT$ is sufficient for advanced stochastics.

In this way, our axiomatic approach gives rise to an alternative way to build
nonstandard real numbers in which, unlike Nelson's approaches \cite{Nelsonist}%
, \cite{REPT} the infinitesimals are not postulated through a new undefined
symbol, but by the existence of magnitudes. It does not have the force of
$REPT$ but has the advantage of being able to deal with some external sets.
Indeed, complete arithmetical solids in a sense incorporate certain external
sets in the form of elements, in particular external sets relevant for
calculatory aspects. Our approach enters also in the tradition of the usual
axiomatic presentation of the real numbers by field axioms, an axiom on
completeness, and possibly an axiom on the Archimedean property.

\section{The Axioms\label{Section Axioms}}

The axioms will come in groups. The first group deals with algebraic
properties. The algebraic axioms consist of axioms for addition, axioms for
multiplication, axioms for the order, axioms relating addition and
multiplication, axioms of existence guaranteeing among other things that
models are richer than fields and axioms on the value of products of
magnitudes. Then we present an axiom scheme on a generalized Dedekind
completion and finally a group of axioms, including a scheme, on natural
numbers and the Archimedean property.

We will present the axioms in a first-order language. Addition, multiplication
and order will be presented in the language $\{+,\cdot,\leq \}$, later on we
add a unary predicate $N$ to deal with natural numbers. Neutrices are
represented by \emph{magnitudes} which are individualized neutral elements. As
such the individualized neutral elements are unique. The proof is similar to
the proof of the uniqueness of neutral elements in groups (see \cite{dinisberg
2011}). Often it is convenient to use the functional notation $e(x)$ to
indicate the individualized neutral element of the element $x$. The
individualized neutral elements for multiplication (unities) are also unique
and we may use the functional notation $u(x)$. With respect with the
individualized neutral element the symmetrical element is also unique. We may
denote it by $s(x)$ or $-x$ in the case of addition and $d\left(  x\right)  $
or $1/x$ in the case of multiplication.

\subsection{Algebraic axioms\label{Subsection Algebraic}}

Axioms for addition and multiplication are similar and gave rise to the notion
of \emph{assembly} in \cite{dinisberg 2011}. An assembly is a completely
regular semigroup (union of groups), in which the magnitude operation is
linear. A structure satisfying Axioms \ref{assemblyassoc} -
\ref{scheiding neutrices} was called a \emph{solid} in \cite{dinisberg 2015
-1}.

\begin{enumerate}
\item[1.] \textbf{Axioms for addition.}
\end{enumerate}

\begin{axiom}
\label{assemblyassoc}$\forall x\forall y\forall z(x+\left(  y+z\right)
=\left(  x+y\right)  +z).$
\end{axiom}

\begin{axiom}
\label{assemblycom}$\forall x\forall y(x+y=y+x).$
\end{axiom}

\begin{axiom}
\label{assemblyneut}$\forall x\exists e\left(  x+e=x\wedge \forall f\left(
x+f=x\rightarrow e+f=e\right)  \right)  .$
\end{axiom}

\begin{axiom}
\label{assemblysim}$\forall x\exists s\left(  x+s=e\left(  x\right)  \wedge
e\left(  s\right)  =e\left(  x\right)  \right)  .$
\end{axiom}

\begin{axiom}
\label{assemblye(xy)}$\forall x\forall y\left(  e\left(  x+y\right)  =e\left(
x\right)  \vee e\left(  x+y\right)  =e\left(  y\right)  \right)  .$
\end{axiom}

It follows from Axiom \ref{assemblyneut} and Axiom \ref{assemblye(xy)} that
$e\left(  x+y\right)  =e\left(  x\right)  +e\left(  y\right)  $, i.e. the
magnitude operation is linear.

\begin{enumerate}
\item[2.] \textbf{Axioms for multiplication}
\end{enumerate}

\begin{axiom}
\label{axiom assoc mult}$\forall x\forall y\forall z(x\left(  yz\right)
=\left(  xy\right)  z).$
\end{axiom}

\begin{axiom}
\label{axiom com mult}$\forall x\forall y(xy=yx).$
\end{axiom}

\begin{axiom}
\label{axiom neut mult}$\forall x\neq e\left(  x\right)  \exists u\left(
xu=x\wedge \forall v\left(  xv=x\rightarrow uv=u\right)  \right)  .$
\end{axiom}

\begin{axiom}
\label{axiom sym mult}$\forall x\neq e\left(  x\right)  \exists d\left(
xd=u\left(  x\right)  \wedge u\left(  d\right)  =u\left(  x\right)  \right)
.$
\end{axiom}

\begin{axiom}
\label{axiom u(xy)}$\forall x\neq e\left(  x\right)  \forall y\neq e\left(
y\right)  \left(  u\left(  xy\right)  =u\left(  x\right)  \vee u\left(
xy\right)  =u\left(  y\right)  \right)  .$
\end{axiom}

Again we have $u\left(  xy\right)  =u\left(  x\right)  u\left(  y\right)  $.

\begin{enumerate}
\item[3.] \textbf{Order axioms}
\end{enumerate}

Axioms \ref{(OA)reflex}-\ref{(OA)total} state that "$\leq$" is a total order
relation. Axiom \ref{(OA)compoper}\ states that the order relation is
compatible with addition. The last two axioms state that the order relation is
compatible with the multiplication by positive elements. With respect with
classical order axioms, essentially the only new axiom is Axiom
\ref{Axiom e(x)maiory}. This axiom states that if an element is "small", in
the sense that if it gets absorbed when added to a certain magnitude, then it
is also smaller than that magnitude in terms of the order.

\begin{axiom}
\label{(OA)reflex}$\forall x(x\leq x).$
\end{axiom}

\begin{axiom}
\label{(OA)antisym}$\forall x\forall y(x\leq y\wedge y\leq x\rightarrow x=y).
$
\end{axiom}

\begin{axiom}
\label{(OA)trans}$\forall x\forall y\forall z(x\leq y\wedge y\leq z\rightarrow
x\leq z).$
\end{axiom}

\begin{axiom}
\label{(OA)total}$\forall x\forall y(x\leq y\vee y\leq x).$
\end{axiom}

\begin{axiom}
\label{(OA)compoper}$\forall x\forall y\forall z\left(  x\leq y\rightarrow
x+z\leq y+z\right)  .$
\end{axiom}

\begin{axiom}
\label{Axiom e(x)maiory}$\forall x\forall y\left(  y+e(x)=e(x)\rightarrow
\left(  y\leq e(x)\wedge-y\leq e(x)\right)  \right)  .$
\end{axiom}

\begin{axiom}
\label{compat mult}$\forall x\forall y\forall z\left(  \left(  e\left(
x\right)  <x\wedge y\leq z\right)  \rightarrow xy\leq xz\right)  .$
\end{axiom}

\begin{axiom}
\label{Axiom Amplification}$\forall x\forall y\forall z\left(  \left(
e\left(  y\right)  \leq y\leq z\right)  \rightarrow e\left(  x\right)  y\leq
e\left(  x\right)  z\right)  .$
\end{axiom}

\begin{enumerate}
\item[4.] \textbf{Axioms concerning addition and multiplication}
\end{enumerate}

The first three axioms state properties of magnitudes. Axiom
\ref{Axiom escala} states that the product of an element and a magnitude is a
magnitude. Axiom \ref{Axiom e(xy)=e(x)y+e(y)x} gives the magnitude of the
product and Axiom \ref{e(u(x))=e(x)d(x)} the magnitude of the individualized
unity. Axiom \ref{Axiom distributivity} states that the distributive law holds
up to a magnitude. Due to this restriction one needs to specify the symmetric
of product as done in Axiom \ref{Axiom s(xy)=s(x)y}.

\begin{axiom}
\label{Axiom escala}$\forall x\forall y\exists z(e(x)y=e(z)).$
\end{axiom}

\begin{axiom}
\label{Axiom e(xy)=e(x)y+e(y)x}$\forall x\forall y\left(
e(xy)=e(x)y+e(y)x\right)  .$
\end{axiom}

\begin{axiom}
\label{e(u(x))=e(x)d(x)}$\forall x\neq e(x)\left(  e(u(x))=e(x)/x\right)  .$
\end{axiom}

\begin{axiom}
\label{Axiom distributivity}$\forall x\forall y\forall z\left(  xy+xz=x\left(
y+z\right)  +e\left(  x\right)  y+e\left(  x\right)  z\right)  .$
\end{axiom}

\begin{axiom}
\label{Axiom s(xy)=s(x)y}$\forall x\forall y\left(  -(xy)=(-x)y\right)  .$
\end{axiom}

\begin{enumerate}
\item[5.] \textbf{Axioms of existence}
\end{enumerate}

Axioms \ref{Axiom neut min} gives the existence of a minimal magnitude which
we will denote by $0$. Elements $p$ such that $e(p)=0$ are called
\emph{precise}. Axiom \ref{Axiom neut mult} gives the existence of a minimal
unity which we denote by $1$. Axiom \ref{Axiom neut max} states that there
exists a maximal magnitude $M$, in fact, when constructing a model it
corresponds to its domain. Axiom \ref{existencia neutrices} states that there
exist magnitudes other than $0$ and $M$, implying that the domain of the model
can no longer be a field. Axiom \ref{numexterno} states that any element is
the sum of a precise element and a magnitude. Axiom \ref{scheiding neutrices}
states that two magnitudes are separated by an element which is not a
magnitude. Such an element is called \emph{zeroless}. It follows from the
existence of zeroless elements that $1\neq0$, hence also that a solid must
contain a copy of $\mathbb{Q}$.

\begin{axiom}
\label{Axiom neut min}$\exists m\forall x\left(  m+x=x\right)  .$
\end{axiom}

\begin{axiom}
\label{Axiom neut mult}$\exists u\forall x\left(  ux=x\right)  .$
\end{axiom}

\begin{axiom}
\label{Axiom neut max}$\exists M\forall x(e\left(  x\right)  +M=M).$
\end{axiom}

\begin{axiom}
\label{existencia neutrices}$\exists x\left(  e\left(  x\right)  \neq0\wedge
e\left(  x\right)  \neq M\right)  .$
\end{axiom}

\begin{axiom}
\label{numexterno}$\forall x\exists a\left(  x=a+e\left(  x\right)  \wedge
e\left(  a\right)  =0\right)  .$
\end{axiom}

\begin{axiom}
\label{scheiding neutrices}$\forall x\forall y(x=e\left(  x\right)  \wedge
y=e(y)\wedge x<y\rightarrow \exists z(z\neq e(z)\wedge x<z<y)).$
\end{axiom}

\begin{enumerate}
\item[6.] \textbf{Axioms on the product of magnitudes}
\end{enumerate}

Next axiom needs some preparatory definitions. A magnitude $e$ is
\emph{idempotent} if $ee=e$. Let $y$ be an idempotent magnitude such that
$1<y$. An \emph{ideal} $z$ of $y$ is a magnitude such that $z\leq y$ and
$\forall p(e\left(  p\right)  =0\wedge0\leq p<y\rightarrow pz\leq z$. An ideal
$x$ of $y$ is \emph{maximal} if $x<y$ and all ideals $z$ such that $x\leq
z\leq y$ satisfy $z=x$ or $z=y$.

In a semantic setting, the ideals defined above are elements and not sets,
such as in the usual algebraic interpretation of ideals of a ring. As will be
shown, the two notions of ideal are closely related. Maximal ideals happen to
be idempotent. The existence of maximal ideals in the setting of rings is
equivalent to the Axiom of Choice. The existence of maximal ideals in terms of
magnitudes will be a consequence of Axiom \ref{Axiom Dedekind completeness} below.

By Axiom \ref{Axiom escala} the product of magnitudes is a magnitude. The
value of the product is obtained by relating them to idempotent magnitudes.
Axiom \ref{Axiom scale to ring} states that a magnitude is the product of a
precise element and an idempotent magnitude. As it turns out, the value of all
products of idempotent magnitudes is determined by Axiom
\ref{Axiom Maximal Ideal}.

\begin{axiom}
\label{Axiom Maximal Ideal}Let $y$ be an idempotent magnitude such that $1<y$
and $x$ be the maximal ideal of $y$. Then $xy=x$.
\end{axiom}

\begin{axiom}
\label{Axiom scale to ring}$\forall x(x=e\left(  x\right)  \rightarrow \exists
p\exists y(e(p)=0\wedge y=e(y)\wedge yy=y\wedge x=py))$.
\end{axiom}

\subsection{Generalized Completeness axiom}

Axiom \ref{Axiom Dedekind completeness} gives a generalized notion of Dedekind
completeness for lower halflines of precise elements. In fact, the set of
precise elements which leave the halfline invariant defines a magnitude $e$;
so like magnitudes, halflines typically are stable under some additions. The
axiom takes the form of a scheme and states that a definable halfline has a
sort of lowest upper bound which is the sum of a precise element and $e$.
We will extend the completeness property in Section
\ref{Section Dedekind completeness} to halflines of non-precise elements. As
will be shown this generates three types of halflines instead of two.

\begin{axiom}
[Generalized Dedekind completeness]\label{Axiom Dedekind completeness} Let $A$
be a formula (possibly with non-precise parameters) allowing for a free
precise variable $x$ and quantifications only over precise variables, and such
that
\begin{equation}
\exists xA\left(  x\right)  \wedge \forall x\left(  A\left(  x\right)  \wedge
y<x\rightarrow A\left(  y\right)  \right)  \label{Condition A}%
\end{equation}
Then one of the following holds:
\end{axiom}

\begin{enumerate}
\item \label{Dedekind1}$\exists \sigma \forall x(A\left(  x\right)
\leftrightarrow x\leq \sigma)$.

\item \label{Dedekind2}$\exists \tau \forall x(A\left(  x\right)
\leftrightarrow \forall t(t+e(\tau)=\tau \rightarrow x<t)$.
\end{enumerate}

It will be shown that \ref{Dedekind1} and \ref{Dedekind2} are mutually
exclusive, and that $\sigma$ and $\tau$ are unique. They are called the
\emph{weak least upper bound} of $A$ and are denoted by $\operatorname*{zup}%
A$. Condition (\ref{Condition A}) expresses the lower-halfline property. If
$A$ is an arbitrary non-empty property, one may define $A^{\prime}$ by%
\[
A^{\prime}(x)\leftrightarrow \exists y(e(y)=0\wedge x\leq y\wedge A(y)).
\]
Then $A^{\prime}$ satisfies (\ref{Condition A}). We extend the notion of weak
supremum by defining $\operatorname*{zup}A=\operatorname*{zup}A^{\prime}$.
Working with upper halflines one may define in a similar way a \emph{weak
greatest lower bound} $\operatorname*{winf}$. It will be seen that both
notions can be appropriately extended to formulas of non-precise variables. We
use this possibility in the following.

We define $\Phi(e)$ respectively $\Psi(f)$ by
\[
e+e=e\wedge e<1\text{.}%
\]%
\[
f+f=f\wedge1<f\text{.}%
\]
Then we define%
\begin{equation}%
\begin{array}
[c]{c}%
\oslash=\operatorname*{zup}\Phi \\
\pounds =\operatorname*{winf}\Psi
\end{array}
. \label{Zerobar El}%
\end{equation}

So $\pounds $ is the minimal magnitude greater than $1$ and $\oslash$ is the
maximal magnitude less than $1$. We will see in Section
\ref{Section Limited numbers infinitesimals}\ that $\pounds $ and $\oslash$
are idempotent and that $\oslash$ is the maximal ideal of $\pounds $ in the
sense of Axiom \ref{Axiom Maximal Ideal}. It results from this axiom that
$\oslash \pounds =\oslash$.

\subsection{Arithmetical axioms}

The last group of axioms allows to distinguish between non-Archimedean ordered
structures and structures with a (nonstandard) archimedean property. We extend
the language with a symbol $N$ which is an unary predicate allowing for a free
precise variable $x$. The symbol $N$ is intended to represent the natural
numbers. In this sense Axiom \ref{Axiom natural numbers} states that there are
no negative natural numbers, $0$ is a natural number, the successor of a
natural number is a natural number, and that these are consecutive indeed,
i.e. between a natural number and its successor there is no other natural number.

\begin{axiom}
[Natural numbers]\label{Axiom natural numbers}%
\begin{align*}
\mathbb{\forall}x\mathbb{(}x  &  <0\rightarrow \lnot N(x))\wedge \\
N(0)\wedge \mathbb{\forall}x(N(x)  &  \rightarrow \mathbb{\forall}%
y\mathbb{(}x<y<x+1\rightarrow \lnot N(y))\wedge N(x+1)).
\end{align*}

\end{axiom}

It is clear that induction does not hold for all formulas. Indeed,
$0<\pounds $ and if $x<\pounds $, then $x+1<\pounds ,$ but there are elements
$x $ such that $\pounds <x$. It is well-known that within nonstandard analysis
one can only apply induction to the so-called internal formulas in the sense
of \cite{Nelsonist}. This means in our context that all parameters must be
natural numbers and also that all references to non-precise elements such as
$\pounds $ and $\oslash$ must be banned. To do so we allow induction in one
precise variable, only for properties with quantifications over precise
variables and with natural numbers as possible parameters.

\begin{axiom}
[Induction]\label{Axiom *s-induction}Let $A$ be a property expressed with the
symbols $0$, $1$, $+$ and $\cdot$, allowing for a free precise variable $x$
and quantifications only over precise variables with all its parameters $y$
satisfying $N(y)$. Then%
\[
(A(0)\wedge \mathbb{\forall}x\mathbb{(}N(x)\rightarrow(A(x)\rightarrow
A(x+1)))\rightarrow \mathbb{\forall}x\mathbb{(}N(x)\rightarrow A(x)).
\]

\end{axiom}

In Section \ref{Section Characterization properties} it is shown that for a
larger class of formulas induction holds over the natural numbers less than
$\pounds $.

The last axiom states the Archimedean property for the natural numbers given
by Axiom \ref{Axiom natural numbers}.

\begin{axiom}
[Archimedean property]\label{axiom archimedes2}%
\[
\forall x\forall y(0<x<y\rightarrow \exists z(N(z)\wedge zx>y)).
\]

\end{axiom}

A set $S$ satisfying all the axioms given above will be called a
\emph{complete arithmetical solid}.

\section{Generalized Dedekind
completeness\label{Section Dedekind completeness}}

Dedekind completeness is the property that every Dedekind cut of the real
numbers is generated by a real number, corresponding to the intuition that the
real line has no \textquotedblleft gaps\textquotedblright. Axiom
\ref{Axiom Dedekind completeness} gives a generalization of the Dedekind
completeness property. Because it is not excluded that some nonstandard models
of the real line do have gaps \cite{KeislerSchmerl}, it is written in the form
of an axiom scheme in a first-order language, i.e. without recurring to
subsets. The axiom scheme is stated for properties of precise elements. In
fact generalized completeness may be extended to properties of arbitrary
elements. As we will see this generates three types of halflines instead of two.

We will call a solid \emph{complete} if it satisfies the Axioms
\ref{assemblyassoc}-\ref{Axiom Dedekind completeness}. Let $S$ be a complete
solid. Let $A$ be a formula of the variable $x$. By a matter of convenience we
identify $A$ with its interpretation in the set $S$. So $x\in A$ is the
interpretation of $A\left(  x\right)  $ and $x\notin A$ the interpretation of
$\lnot A\left(  x\right)  $.

\begin{definition}
\label{Definition halfline}Let $S$ be a complete solid and let $\emptyset \neq
A,B\subseteq S$. Then $A$ is said to be a \emph{lower halfline }if $x\in A$
and $y<x$ imply that $y\in A$ and $B$ is said to be an \emph{upper halfline}
if $x\in B$ and $x<y$ imply that $y\in B$. A lower halfline $A$ is
\emph{precise} if there is no precise positive $d$ such that $a+d\in A$ for
all $a\in A$.
\end{definition}

\begin{remark}
\label{Complementar halfline}If $A\neq S$ is a lower halfline then
$B=S\backslash A\neq \emptyset$ is an upper halfline and vice-versa.
\end{remark}

\begin{theorem}
\label{Theorem Dedekind}If $A$ is a definable lower halfline in a complete
solid $S$, it has one of the following forms:
\end{theorem}

\begin{enumerate}
\item \label{Theorem Dedekind 1}$\exists \rho \forall x(x\in A\leftrightarrow
x\leq \rho).$

\item \label{Theorem Dedekind 2}$\exists \sigma \forall x(x\in A\leftrightarrow
x<\sigma).$

\item \label{Theorem Dedekind 3}$\exists \tau \forall x(x\in A\leftrightarrow
\forall t(t+e(\tau)=\tau \rightarrow x<t).$
\end{enumerate}

In order to prove Theorem \ref{Theorem Dedekind}\ we will make Axiom
\ref{scheiding neutrices}\ more operational, by showing that two magnitudes
can always be separated by a precise element. Also the elements $t$ of Axiom
\ref{Axiom Dedekind completeness}.\ref{Dedekind2} and Theorem
\ref{Theorem Dedekind}.\ref{Theorem Dedekind 3} may be taken precise. We prove
also some other properties on separation by a precise element.

\begin{proposition}
\label{t precise}In Axiom \ref{Axiom Dedekind completeness}.\ref{Dedekind2}
the elements $t$ may be taken precise.
\end{proposition}

\begin{proof}
Let $t$ be such that $t+e(\tau)=\tau$. Then $t=p+e\left(  t\right)  $ with
$e\left(  p\right)  =0$. Because $e\left(  t\right)  \leq e(\tau)$, one has
$p+e(\tau)=p+e\left(  t\right)  +e(\tau)=t+e(\tau)=\tau$.

Let $x$ be precise. Suppose that $x<t$ for all $t$ such that $t+e(\tau)=\tau$.
Then it holds in particular for $t=p$ with $p$ precise. Conversely, suppose
that $x<p$, for all precise $p$ such that $\tau=p+e\left(  \tau \right)  $. Let
$t$ be such that $t+e(\tau)=\tau$. Then there exists a precise $q$ such that
$t=q+e\left(  t\right)  $. Then $q+e(\tau)=\tau$. Hence $x<q\leq q+e\left(
t\right)  =t$. This implies that the two criteria are equivalent.
\end{proof}

\begin{proposition}
\label{t precise theorem}In Theorem \ref{Theorem Dedekind}%
.\ref{Theorem Dedekind 3} the elements $t$ may be taken precise.
\end{proposition}

The proof is analogous to the proof of Proposition \ref{t precise}.

\begin{lemma}
\label{precise separation}Let $S$ be a solid and let $x,z\in S$. If $e\left(
x\right)  <z$ and $z$ is zeroless, there is a precise element $t$ such that
$e\left(  x\right)  <t\leq t+e\left(  z\right)  <z$.
\end{lemma}

\begin{proof}
Let $z=p+e\left(  z\right)  $ with $p$ precise. If $z<e\left(  z\right)  $
then $z<e\left(  x\right)  $, by \cite[Prop. 2.10]{dinisberg 2015 -2}, a
contradiction. Hence $e\left(  z\right)  <z$, meaning that $p$ is positive, so
$0<p/2<p$. One has $e\left(  z\right)  <p/2$. Indeed, if $p/2\leq e\left(
z\right)  $ one would have $p\leq2e\left(  z\right)  =e\left(  z\right)  $
which is a contradiction. It follows that $p/2+e\left(  z\right)  <p\leq
p+e\left(  z\right)  =z$. If $e\left(  x\right)  \leq e\left(  z\right)  $, we
are done. If $e\left(  z\right)  <e\left(  x\right)  $, suppose that $p/2\leq
e(x)$. Then $p\leq2e(x)=e(x)$ and $z=p+e(z)\leq e(x)+e(x)=e(x)$, a
contradiction. Hence $e\left(  x\right)  <p/2\leq p/2+e\left(  z\right)  $.
\end{proof}

\begin{lemma}
\label{Precise separation neutrices}The element $z$ in Axiom
\ref{scheiding neutrices} may be supposed precise.
\end{lemma}

\begin{proof}
Let $z$ be zeroless and such that $e\left(  x\right)  <z<e(y)$. By Lemma
\ref{precise separation} there exists a precise $t$ such $e\left(  x\right)
<t<z<e(y)$.
\end{proof}

\begin{lemma}
\label{Precise separation elements}Let $S$ be a solid and $x,y\in S$ be such
that $x<y$. Then there exists a precise element $p$ such that $x<p<y$.
\end{lemma}

\begin{proof}
Without restriction of generality we may assume that $x=e(x)$. The case
$y=e(y)$ follows from Lemma \ref{Precise separation neutrices} and the case
where $y$ is zeroless follows from Lemma \ref{precise separation}.
\end{proof}

\begin{lemma}
\label{Precise separation from hole}Let $S$ be a solid and $x,y\in S$ be such
that $\forall t(e(t)=0\wedge t+e(y)=y\rightarrow x<t)$. Then there exists a
precise element $p$ such that $x<p$ and $\forall t(e(t)=0\wedge
t+e(y)=y\rightarrow p<t)$.
\end{lemma}

\begin{proof}
Without restriction of generality we may assume that $x=e(x)$. Then $y$ is
zeroless. Let $q$ be precise such that $y=q+e(y)$. Put $p=q/2$. Then $e(x)<p$,
otherwise $q<2e(x)=e(x)$, a contradiction. Also $e(y)<p$, otherwise
$q<2e(y)=e(y)$ and $y$ would not be zeroless. Let $t$ be precise such that
$t+e(y)=y$. Then $e(y)<t$, otherwise $y$ would not be zeroless. Suppose $t\leq
p$. Then $y=t+e(y)<t+t\leq q$, a contradiction. Hence $p<t$.
\end{proof}

\bigskip

\begin{proof}
[Proof of Theorem \ref{Theorem Dedekind}]Let $A$ be a definable lower halfline
in a complete solid $S$. We define $\hat{A}$ by $\hat{A}=\left \{  x\in
A\left \vert e\left(  x\right)  =0\right.  \right \}  $. By Axiom
\ref{Axiom Dedekind completeness} and Proposition \ref{t precise}, either
there exists $\alpha \in S$ such that, whenever $p$ is precise, one has $p\in
A$ if and only if $p\leq \alpha$, or there exists $\beta$ such that whenever
$p$ is precise, one has $p\in A$ if and only if $p<q$ for all precise $q$ such
that $q+e(\beta)=\beta$.

As for the first case, we distinguish the subcases $\alpha \in A$ and
$\alpha \not \in A$. Assume $\alpha \in A$. Let $x\in A$. Suppose $\alpha<x$. By
Lemma \ref{Precise separation elements} there exists a precise element $r\in
A$ such that $\alpha<r<x$, a contradiction. Hence $x\leq \alpha$. Conversely,
assume $x\leq \alpha$. Then $x\in A$ by the definition of lower halfline.
Taking $\rho=\alpha$ we obtain $x\in A$ if and only if $x\leq \rho$.

Assume now that $\alpha \not \in A$. One proves as above that if $x\in A$ then
$x\leq \alpha$, in fact, $x<\alpha$ because $\alpha \not \in A$. Conversely,
assume $x<\alpha$. By Lemma \ref{Precise separation elements} there exists a
precise element $r$ such that $x<r<\alpha$. Then $r\in \hat{A}\subseteq A$.
Hence $x\in A$. Taking $\sigma=\alpha$ we obtain $x\in A$ if and only if
$x\leq \sigma$.

As for the second case, assume first that $x\in A$. Suppose that there exists
a precise $t$ such that $t<x$ and $t+e(\beta)=\beta$. We may write
$x=p+e\left(  x\right)  $ with $p$ precise and $t<p$. Now $p<t$ because
$p\in \hat{A}$, a contradiction. Hence $x<t$. Finally assume that $x<t$ for all
precise $t$ such that $t+e(\beta)=\beta$. By Lemma
\ref{Precise separation from hole} there exists a precise $p$ such that $x<p$
and $p<t$ for all precise $t$ such that $t+e(\beta)=\beta$. Then $p\in \hat
{A}\subseteq A$, hence $x\in A$ because $A$ is a lower halfline. Taking
$\tau=\beta$, we conclude that $x\in A$ if and only if $\forall t(e(t)=0\wedge
t+e(\tau)=\tau \rightarrow x<t)$. By Proposition \ref{t precise theorem} this
is equivalent to $\forall t(t+e(\tau)=\tau \rightarrow x<t)$.
\end{proof}

If $A$ is a precise lower halfline the case \ref{Theorem Dedekind}%
.\ref{Theorem Dedekind 3} reduces to the case \ref{Theorem Dedekind}%
.\ref{Theorem Dedekind 2}, so the Generalized Dedekind completeness of Axiom
\ref{Axiom Dedekind completeness} and Theorem \ref{Theorem Dedekind}
correspond to ordinary Dedekind completeness.

\begin{proposition}
\label{Dedekind precise}If $A$ is precise, in the third case of the criterion
in Theorem \ref{Theorem Dedekind} the element $\tau$ is precise. In fact the
criterion is equivalent to
\begin{equation}
\exists \tau(e\left(  \tau \right)  =0\wedge \forall x(x\in A\leftrightarrow
x<\tau). \label{criterium 3 simplified}%
\end{equation}

\end{proposition}

\begin{proof}
Suppose that $0<e\left(  \tau \right)  $. Let $p$ be precise and such that
$0<p<e\left(  \tau \right)  $. Let $x\in A$. If there exists $t$ such that
$t+e(\tau)=\tau$ and $t\leq x+p$, then $\tau \leq x+p+e(\tau)=x+e(\tau)$. Note
that $x+e(\tau)\leq t+e(\tau)=\tau$. Hence $x+e(\tau)=\tau$, which means that
$x<x$, a contradiction. Hence $x+p\in A$, but this means that $A$ is not
precise, again a contradiction. Hence $e\left(  \tau \right)  =0$ and $\tau$ is
precise. In addition, if $t$ is such that $t+e(\tau)=\tau$, then $t=\tau$
\end{proof}

If the lower halfline $A$ is not precise, the three cases may indeed occur and
are mutually exclusive. We will call the elements $\rho,\sigma$ and $\tau$
\emph{weak least upper bounds}, denoted by $\operatorname*{zup}A$. Moreover,
the elements $\rho,\sigma$ and $\tau$ are unique and we write $A=\left(
-\infty,\rho \right]  $, $A=\left(  -\infty,\sigma \right)  $ and $A=\left(
-\infty,\tau \right[  [$ respectively. In the first case the halfline is called
\emph{closed}, and $\rho$ may be called the \emph{maximum} of $A$, written
$\rho \equiv \max A$. In the second case the halfline is called \emph{open} and
$\sigma$ is an ordinary least upper bound, which we may call the
\emph{supremum }of $A$, written $\sigma \equiv \sup A$. In the third case we
call the halfline \emph{strongly open} (see also \cite{Wallet}) and $\tau$ the
\emph{weak supremum} of $A$. We may define weak least upper bounds for any set
$A$ by defining $\operatorname*{zup}A\equiv \operatorname*{zup}\overline{A}$
where $\overline{A}\equiv \{x\in S|\exists a\in A(x\leq a)\}$.

\begin{theorem}
\label{Theorem three cases lower}With respect to Theorem
\ref{Theorem Dedekind} the elements $\rho,\sigma$ and $\tau$ are unique, and
the cases \ref{Theorem Dedekind 1} and \ref{Theorem Dedekind 2}, and the cases
\ref{Theorem Dedekind 1} and \ref{Theorem Dedekind 3} are mutually exclusive.
If $A$ is a precise lower halfline the third case reduces to the second case.
If $A$ is not precise the three cases are mutually exclusive. Moreover,
$\rho \in A$ and $\sigma,\tau \notin A$.
\end{theorem}

\begin{proof}
Clearly $\rho \in A$ and $\sigma \notin A$, because $\rho \leq \rho$ and
$\sigma \not <  \sigma$. Suppose towards a contradiction that $\tau \in A$. Then
for all $t$ such that $t+e\left(  \tau \right)  =\tau$ one has $\tau<t$. Then
$\tau<t+e\left(  \tau \right)  =\tau$, a contradiction. Hence $\tau \notin A$.

To show that $\rho$ is unique suppose that $\rho^{\prime}$ is such that $x\in
A$ if and only if $x\leq \rho$ and if and only if $x\leq \rho^{\prime}$. Then
$\rho,\rho^{\prime}\in A$, hence $\rho \leq \rho^{\prime}$ and $\rho^{\prime
}\leq \rho$. One concludes that $\rho$ is unique by Axiom \ref{(OA)antisym}.

To show that $\sigma$ is unique suppose that $\sigma^{\prime}$ is such that
$x\in A$ if and only if $x<\sigma$ if and only if $x<\sigma^{\prime}$. If
$\sigma<\sigma^{\prime}$ then $\sigma \in A$ hence $\sigma<\sigma$ which is
absurd. If $\sigma^{\prime}<\sigma$ then similarly $\sigma^{\prime}\in A$ and
$\sigma^{\prime}<\sigma^{\prime}$, which is absurd. Hence $\sigma
=\sigma^{\prime}$ by Axiom \ref{(OA)total}.

In order to show that $\tau$ is unique, suppose that $x\in A$ if and only if
$x<t$ for all precise $t$ with $t+e(\tau)=\tau$, and also if and only if
$x<t^{\prime}$ for all precise $t^{\prime}$ with $t^{\prime}+e(\tau^{\prime
})=\tau^{\prime}$. Assume that $\tau<\tau^{\prime}$. We may suppose that
$\tau=e\left(  \tau \right)  $. Suppose first that $\tau^{\prime}=e\left(
\tau^{\prime}\right)  $. Then by Lemma \ref{precise separation} there is a
precise element $p$ such that $e\left(  \tau \right)  <p<\tau^{\prime}$. Then
$-p<t$ for all precise $t$ such that $t+e(\tau)=\tau$, otherwise $p+t\leq0\leq
e\left(  \tau \right)  $, so $p\leq-t+e\left(  \tau \right)  =-\left(
t-e\left(  \tau \right)  \right)  =-\left(  t+e\left(  \tau \right)  \right)
=-e\left(  \tau \right)  =e\left(  \tau \right)  $, which is a contradiction.
Hence $-p\in A$. On the other hand $-p+e\left(  \tau^{\prime}\right)
=-\left(  p-e\left(  \tau^{\prime}\right)  \right)  =-\left(  p+e\left(
\tau^{\prime}\right)  \right)  =-e\left(  \tau^{\prime}\right)  =e\left(
\tau^{\prime}\right)  $. Hence $-p\notin A$, a contradiction. Secondly, we
suppose that $\tau^{\prime}$ is zeroless. It follows from Lemma
\ref{precise separation} that there exists a precise $p$ such that $e\left(
\tau \right)  <p+e\left(  \tau^{\prime}\right)  <\tau^{\prime}$. By the second
inequality $p\in A$ and by the first inequality $p\notin A$, a contradiction.
Hence $\tau^{\prime}\leq \tau$. Similarly one shows that $\tau \leq \tau^{\prime
}$. Hence $\tau=\tau^{\prime}$.

We prove next that the cases \ref{Theorem Dedekind}.\ref{Theorem Dedekind 1}
and \ref{Theorem Dedekind}.\ref{Theorem Dedekind 2} are mutually exclusive. If
not, because $\rho \in A$ one has $\rho<\sigma$. By Lemma
\ref{precise separation} there is a precise element $s$ such that
$\rho<s<\sigma$. Then $s\notin A$ because $\rho<s$, and $s\in A$ because
$s<\sigma$, a contradiction. Hence the cases \ref{Theorem Dedekind}%
.\ref{Theorem Dedekind 1} and \ref{Theorem Dedekind}.\ref{Theorem Dedekind 2}
are mutually exclusive.

We prove next that the cases \ref{Theorem Dedekind}.\ref{Theorem Dedekind 1}
and \ref{Theorem Dedekind}.\ref{Theorem Dedekind 3} are mutually exclusive. If
not, because $\rho \in A$, for all $t$ such that $t+e\left(  \tau \right)
=\tau$ one has $\rho<t$. By Lemma \ref{Precise separation from hole} there is
a precise $p$ such that $\rho<p$ and $p<t$ for all $t$ such that $t+e\left(
\tau \right)  =\tau$. Then both $p\notin A$ and $p\in A$, a contradiction.

Finally we relate the cases \ref{Theorem Dedekind}.\ref{Theorem Dedekind 2}
and \ref{Theorem Dedekind}.\ref{Theorem Dedekind 3}. By Proposition
\ref{Dedekind precise} they coincide if $A$ is precise. Assume that $A$ is not
precise. If $\sigma<\tau$, we obtain a contradiction along the lines of the
previous case. Suppose $\tau<\sigma$. By Lemma
\ref{Precise separation elements} there exists a precise $q$ such that
$\tau<p<\sigma$. Then both $p\notin A$ and $p\in A$, a contradiction. Hence
$\sigma=\tau$. Hence the cases \ref{Theorem Dedekind}.\ref{Theorem Dedekind 2}
and \ref{Theorem Dedekind}.\ref{Theorem Dedekind 3} are mutually exclusive.
\end{proof}

As for upper halflines, the definition of \emph{weak greatest lower bounds}
$\operatorname*{winf}$ is similar, but not entirely analogous, to the
definition of weak least upper bounds, according to the three possibilities
mentioned in the following theorem.

\begin{theorem}
\label{Theorem Dedekind upper}Let $S$ be a complete solid and $B\subseteq S$
be a definable upper halfline. Then $B$ has one of the following forms:
\end{theorem}

\begin{enumerate}
\item \label{Theorem Dedekind upper1}$\exists \rho \forall y(y\in
B\leftrightarrow \rho \leq y).$

\item \label{Theorem Dedekind upper2}$\exists \sigma \forall y(y\in
B\leftrightarrow \sigma<y).$

\item \label{Theorem Dedekind upper3}$\exists \tau \forall y(y\in
B\leftrightarrow \exists t(t+e(\tau)=\tau \wedge t\leq y)).$
\end{enumerate}

\begin{proof}
The case $B=S$ corresponds to \ref{Theorem Dedekind upper}%
.\ref{Theorem Dedekind upper3}. If $B\subset S$, define $A=S\backslash B$.
Then $A$ is a lower halfline. Let $\zeta={\operatorname*{zup}}(A)$. If
$\forall x\left(  x\in A\Leftrightarrow x\leq \zeta \right)  $, then $\forall
y\left(  y\in B\Leftrightarrow \zeta<y\right)  $. If $\forall x\left(  x\in
A\Leftrightarrow x<\zeta \right)  $, then $\forall y\left(  y\in
B\Leftrightarrow \zeta \leq y\right)  $. Finally, if $\forall x(x\in
A\Leftrightarrow \forall z(z+e\left(  \zeta \right)  =\zeta \Rightarrow x<z))$,
then $\forall y(y\in B\Leftrightarrow \exists z(z+e\left(  \zeta \right)
=\zeta \wedge z\leq y))$.
\end{proof}

In the first case we call the upper halfline $B$ \emph{closed} with
\emph{minimum} $\rho \equiv \min B$, in the second case we call the upper
halfline \emph{open} with \emph{infimum }$\sigma \equiv \inf B$, and in the
third case we call the upper halfline \emph{strongly open }with \emph{weak
infimum }$\tau=\operatorname*{winf}B$. We may define greatest lower bounds for
any set $B$ by defining $\operatorname*{winf}B\equiv \operatorname*{winf}%
\underline{B}$ where $\underline{B}\equiv \{x\in S|\exists b\in B(b\leq x)\}$.
Note that the complement of a closed lower halfline is open (if not empty),
the complement of a open lower halfline is closed and the complement of a
non-precise strongly open lower halfline is again strongly open. As a
corollary to Theorem \ref{Theorem three cases lower} we obtain the following
criterion for upper halflines.

\begin{corollary}
Let $S$ be a complete solid and $B\subseteq S$ be an upper halfline. The
elements $\rho,\sigma$ and $\tau$ of Theorem \ref{Theorem Dedekind upper} are
unique and the cases \ref{Theorem Dedekind upper1} and
\ref{Theorem Dedekind upper2}, and the cases \ref{Theorem Dedekind upper1} and
\ref{Theorem Dedekind upper3} are mutually exclusive. If $B$ is a precise
lower halfline, it is strongly open if and only if it is closed. If $B$ is not
a precise lower halfline, the properties of being open, closed or strongly
open are mutually exclusive.
\end{corollary}

It is to be noted that also the case of the complete solid $S$ itself enters
in the above classifications. Considered as a lower halfline it has a maximum
in the form of the maximal magnitude $M$, while $M$ acts as a weak infimum, if
$S$\ is considered as an upper halfline.

\begin{proposition}
\label{Zup E magnitude}Let $S$ be a complete solid and $E$ be a set of
magnitudes. Then $\operatorname*{zup}E$ and $\operatorname*{winf}E$ are
magnitudes. In fact, $\operatorname*{zup}E$ is a maximum or a supremum and
$\operatorname*{winf}E$ is a minimum or an infimum.
\end{proposition}

\begin{proof}
Let $Z=\operatorname*{zup}E$. If $Z$ is a maximum it is clearly a magnitude.
Assume that $Z$ is a supremum. If $Z$ is zeroless, then $Z/2<Z$. Then there
must exist an element $f$ such that $f+f=f$ and $Z/2<f<Z$, otherwise $Z/2$
would already be an upper bound of $E$. Hence $Z<2f=f$, a contradiction. Hence
$Z$ is a magnitude.

We show that $Z$ cannot be a weak supremum. If such, $Z$ cannot be a
magnitude, otherwise every element in $Z$ would be negative, in contradiction
with the fact that $0\leq Z$. Also $Z$ cannot be zeroless. Indeed, then $Z$
would be of the form $t+e\left(  Z\right)  $ with $t$ precise and $e\left(
Z\right)  <t$. Then $t/2+e\left(  Z\right)  <Z$. Then there must exist an
element $f$ such that $f+f=f$ and $t/2+e\left(  Z\right)  <f<Z$, otherwise
$t/2+e\left(  Z\right)  $ would already be an upper bound of $E$. Hence
$Z=t+e\left(  Z\right)  <2f=f<Z$, a contradiction. We conclude that $Z$ is a
maximum or a supremum.

The proof for $\operatorname*{winf}E$ is similar.
\end{proof}

\section{Limited numbers and
infinitesimals\label{Section Limited numbers infinitesimals}}

\begin{notation}
Let $S$ be a complete solid. With some abuse of language the
$\operatorname*{winf}$ of the magnitudes larger than $1$ is noted $\pounds $,
i.e. $\pounds \mathcal{\equiv}\operatorname*{winf}\{e\in S|e+e=e\wedge1<e\}$
and the $\operatorname*{zup}$ of the magnitudes smaller than $1$ is noted
$\oslash$, i.e. $\oslash \mathcal{\equiv}\operatorname*{zup}\{e\in
S|e+e=e\wedge e<1\}$.
\end{notation}

\begin{theorem}
\label{zerobar<1<L}Within a complete solid the sets $\oslash$ and $\pounds $
are magnitudes and satisfy
\[
0<\oslash<1<\pounds <M\text{.}%
\]

\end{theorem}

\begin{proof}
Let $E=\left \{  e|e+e=e\wedge e<1\right \}  $. By Axiom
\ref{existencia neutrices} there exists $e$ such that $e+e=e$ and $0<e<M$. If
$1<e$ then there is a precise element $p$ such that $1<e<p<M$. Then
$0<1/p<e/p<1$. If $e<1$ then there is a precise element $q$ such that $0<q<e$.
Then $0<1<e/q<1/q<M$. Hence there exists a magnitude between $0$ and $1$ and a
magnitude between $1$ and $M$. So $0<\oslash$ and $\pounds <M$.

By Proposition \ref{Zup E magnitude}, $\oslash$ is a magnitude. If $\oslash$
is a maximum, clearly $\oslash<1$. If $1<\oslash$, there must exist an element
$e$ such that $e+e=e$ and $1<e<1$, a contradiction. We conclude that
$\oslash<1$.

The proof that $\pounds $ is a magnitude and $1<\pounds $ is analogous, now
using $F=\left \{  e|e+e=e\wedge1<e\right \}  $.
\end{proof}

\begin{corollary}
\label{no magnitude between}There are no magnitudes between the magnitudes
$\oslash$ and $\pounds $.
\end{corollary}

Since $0<\oslash$ and $\pounds <M$, by Axiom \ref{scheiding neutrices} and
Lemma \ref{Precise separation neutrices} there are precise elements $p$ and
$q$ such that $0<p<\oslash$ and and $\pounds <q<M$.

Below we show that $\oslash$ and $\pounds $ are idempotent magnitudes for
multiplication, i.e. $\oslash \oslash=\oslash$ and $\pounds \pounds =\pounds $.
Moreover $\oslash \pounds =\oslash$ or $\oslash \pounds =\pounds $. Indeed, it
follows from distributivity \cite[Cor. 2.29 ]{dinisberg 2015 -2} that the
product of two magnitudes $e$ and $f$ is a magnitude, since $ef+ef=e(f+f)=ef$.
In particular $\oslash \pounds $ is a magnitude. By compatibility with the
ordering
\begin{equation}
\oslash \leq \oslash1\leq \oslash \pounds \leq1\pounds =\pounds .
\label{zerobar L indetermined}%
\end{equation}
To decide whether $\oslash \pounds =\oslash$ or $\oslash \pounds =\pounds $, we
need Axiom \ref{Axiom Maximal Ideal}. In fact $\oslash \pounds =\oslash$, which
will be shown in Section \ref{Section Product magnitudes}. Next lemma states
some basic properties of $\oslash$ and $\pounds $.

\begin{lemma}
\label{Lemma L<p}Let $0<p$ be precise. Then

\begin{enumerate}
\item \label{L<p iff 1/p<zerobar}$\pounds <p$ if and only if $1/p<\oslash$.

\item \label{1/p<L}$\oslash<p$ if and only if $1/p<\pounds $.

\item \label{zerobar<p^2}If $p<\oslash$ then $\sqrt{p}<\oslash$.

\item \label{L<raizp}If $\pounds <p$ then $\pounds <\sqrt{p}$.

\item \label{p2<L}If $\oslash<p<\pounds $, then $\oslash<p^{2}<\pounds $.

\item \label{zerobar sup}$\oslash=\sup \{p|e\left(  p\right)  =0\wedge
\pounds <1/p\}$ and $\pounds =\inf \{1/p|e\left(  p\right)  =0\wedge
p<\oslash \}$.
\end{enumerate}
\end{lemma}

\begin{proof}
\ref{L<p iff 1/p<zerobar}. Assume $\pounds <p$. Then $\pounds /p\leq1$.
Because $\pounds /p$ is a magnitude, we have $\pounds /p<1$, so
$\pounds /p\leq \oslash$. Since $1/p<\pounds /p$ we derive that $1/p<\oslash$.
Assume now that $1/p<\oslash$. Then $1\leq p\oslash$, in fact $1<p\oslash$,
because $p\oslash$ is a magnitude. So $\pounds \leq p\oslash<p\cdot1=p$. Hence
$\pounds <p$ if and only if $1/p<\oslash$.

\ref{1/p<L}. Directly from Part \ref{L<p iff 1/p<zerobar}.

\ref{zerobar<p^2}. Suppose that $p<\oslash$ and $\oslash<\sqrt{p}$. Then
$\oslash/\sqrt{p}<1$. Then $\oslash/\sqrt{p}\leq \oslash$. Hence $p<\oslash
\leq \sqrt{p}\oslash$, so $\sqrt{p}<\oslash$, a contradiction. Hence $\sqrt
{p}<\oslash$.

\ref{L<raizp}. Suppose that $\pounds <p$ and $\sqrt{p}<\pounds $. Then
$1<\pounds /\sqrt{p}$. Then $\pounds \leq \pounds /\sqrt{p}$. Hence $\sqrt
{p}\pounds \leq \pounds <p$, so $\pounds <\sqrt{p}$, a contradiction. Hence
$\pounds <\sqrt{p}$.

\ref{p2<L}. This part is a direct consequence of Part \ref{zerobar<p^2} and
Part \ref{L<raizp}.

\ref{zerobar sup}. Let $0<p$ be precise. Put $\oslash^{\prime}=\sup \left \{
p|e\left(  p\right)  =0\wedge \pounds <1/p\right \}  $. Suppose that
$\oslash<\oslash^{\prime}$. Then there exists a precise element $q$ such that
$\oslash<q<\oslash^{\prime}$. Then $1/q<\pounds $ by Part \ref{1/p<L}, while
$\pounds <1/q$ by definition of $\oslash^{\prime}$. Hence $\oslash^{\prime
}\leq \oslash$. Suppose now that $\oslash^{\prime}<\oslash$. Then there exists
a precise element $r$ such that $\oslash^{\prime}<r<\oslash$. Then
$1/r<\pounds $ by definition of $\oslash^{\prime}$, while $\pounds <1/r$ by
Part \ref{L<p iff 1/p<zerobar}, a contradiction. Hence $\oslash=\oslash
^{\prime}$. The second part is proved in an analogous way.
\end{proof}

\begin{theorem}
\label{Proposition Mult magnitudes}One has
\end{theorem}

\begin{enumerate}
\item \label{zerobar.zerobar}$\oslash \oslash=\oslash$.

\item \label{L.L}$\pounds \pounds =\pounds $.
\end{enumerate}

\begin{proof}
\ref{zerobar.zerobar}. Suppose $\oslash \oslash<\oslash$. Then there is a
precise element $p$ such that $\oslash \oslash<p<\oslash$. By Lemma
\ref{Lemma L<p}.\ref{zerobar<p^2} one has $\sqrt{p}<\oslash$. Then $p=\sqrt
{p}\sqrt{p}<\oslash \oslash$, a contradiction. Hence $\oslash \leq \oslash
\oslash$. Because $\oslash<1$, also $\oslash \oslash \leq \oslash$. We conclude
that $\oslash \oslash=\oslash$.

\ref{L.L}. Suppose $\pounds <\pounds \pounds $. Then there is a precise
element $p$ such that $\pounds <p<\pounds \pounds $. By Lemma \ref{Lemma L<p}%
.\ref{L<raizp} one has $\pounds <\sqrt{p}$. Then $\pounds \pounds <\sqrt
{p}\sqrt{p}=p$, a contradiction. Hence $\pounds \pounds \leq \pounds $. Because
$1<\pounds $, also $\pounds \leq \pounds \pounds $. We conclude that
$\pounds \pounds =\pounds $.
\end{proof}

\section{Product of magnitudes\label{Section Product magnitudes}}

Let $f$ and $g$ be two magnitudes. Though the product $fg$ is well-defined as
a magnitude, the value of this magnitude is not determined. For example,
formula (\ref{zerobar L indetermined}) and Corollary
\ref{no magnitude between}\ show that $\oslash \pounds =\oslash$ or
$\oslash \pounds =\pounds $ but do not decide which equality holds. We will see
that Axiom \ref{Axiom Maximal Ideal} has as a consequence that $\oslash
\pounds =\oslash$. In fact this axiom together with Axiom
\ref{Axiom scale to ring} implies that the value of the product $fg$ is
determined for all magnitudes$\ f$ and $g$. Axiom \ref{Axiom Maximal Ideal}
gives the value of the product of a magnitude which is idempotent for
multiplication with its so-called maximal ideal. Using an order argument, it
will be shown that the axiom determines the value of the product of all
magnitudes which are idempotent for multiplication. Axiom
\ref{Axiom scale to ring} states that every magnitude is a multiple of an
idempotent magnitude which is shown to be unique. This enables to determine
all products of magnitudes.

We recall first some definitions from Section \ref{Section Axioms}%
.\ref{Subsection Algebraic}.

\begin{definition}
A magnitude $I$ is called \emph{idempotent} if $II=I$.
\end{definition}

Clearly $0$ and $M$ are idempotent magnitudes and by Theorem
\ref{Proposition Mult magnitudes}, also $\oslash$ and $\pounds $. Note that if
$e$ and $f$ are idempotent then $ef$ is also idempotent, because $efef=eeff=ef
$.

\begin{definition}
\label{Definition ideal}Let $e$ and $I$ be magnitudes such that $1<I$, $I$ is
idempotent and $e\leq I$. If for all precise positive $q$ such that $q<I$ it
holds that $eq\leq e$ we say that $e$ is an \emph{ideal} of $I$. An ideal $e$
of $I$ is said to be \emph{maximal} if $e<I$ and for every ideal $f$ of $I$
such that $e\leq f\leq I$ one has $e=f$ or $f=I$.
\end{definition}

Every idempotent magnitude $I$ such that $1<I$ possesses an ideal. Indeed, $0$
is an ideal of $I$, for $0$ is an idempotent magnitude and for all precise $q$
such that $q<I$ one has $0q=0$. If $I<M$, then $I$ has nonzero ideals and the
existence of a\ maximal ideal of $I$ will be a consequence of generalized
Dedekind completeness.

\begin{notation}
Unless otherwise said, we let $J$ be an idempotent magnitude such that $1<J<M$
and $I=\sup A$, where $A\mathcal{\equiv}\left \{  1/\omega|\omega \text{
precise, }J<\left \vert \omega \right \vert \right \}  $.
\end{notation}

\begin{theorem}
\label{Thm max ideal}The maximal ideal of $M$ is equal to $0$. If $1<J<M$,
then $0<I\leq \oslash$ and $I$ is the maximal ideal of $J$.
\end{theorem}

The first part of the theorem follows from the fact that $x.0=0$ for all $x\in
S$ \cite[Prop. 3.5]{dinisberg 2015 -2}, and $x.M=M$ for all $x\in S$ such that
$x\neq0$. To prove the remaining part we start with some preparatory lemmas.

\begin{lemma}
\label{Lemma max ideal 1}Let $0<p$ be precise. Then
\end{lemma}

\begin{enumerate}
\item \label{2p<I}If $p<I$ then $2p<I.$

\item \label{sqrt p<I}If $p<I$ then $\sqrt{p}<I.$

\item \label{I<1/p}If $1\leq p<J$ then $I<1/p.$
\end{enumerate}

\begin{proof}
Assume that $p<I$. Then $p\in A$. Hence $J<1/p$. Then $J<1/2p$ because $J$ is
a magnitude and $J<1/\sqrt{p}$ because $J$ is idempotent. Hence $2p<I$ and
$\sqrt{p}<I$. This proves Part \ref{2p<I}\ and \ref{sqrt p<I}. Part
\ref{I<1/p} follows directly from the definition of $I$.
\end{proof}

\begin{lemma}
\label{Lemma max ideal 2}$I<1$ is an idempotent magnitude.
\end{lemma}

\begin{proof}
The fact that $I<1$ follows from Lemma \ref{Lemma max ideal 1}.\ref{I<1/p}. It
follows from Lemma \ref{Lemma max ideal 1}.\ref{2p<I} that $I\leq I+I$.
Suppose that $I<I+I$. Then there exists a precise element $p$ such that
$I<p<I+I$. Then $p=p/2+p/2<I$ by Lemma \ref{Lemma max ideal 1}.\ref{2p<I}, a
contradiction. Hence $I$ is a magnitude. To show that $I$ is idempotent,
observe that $I.I\leq I$ because $I<1$. Suppose $I.I<I$. Then there exists a
precise element $p$ such that $I.I<p<I$. Then $p=\sqrt{p}\sqrt{p}<I$ by Lemma
\ref{Lemma max ideal 1}.\ref{sqrt p<I}, a contradiction. Hence $I$ is idempotent.
\end{proof}

\begin{lemma}
\label{Lemma max ideal 4}Let $r,p$ be precise and such that $1<p<J$ and $r<I$.
Then $rp<I$.
\end{lemma}

\begin{proof}
By the definition of $I$ one has $J<1/r$. Suppose that $1/\left(  rp\right)
=\left(  1/r\right)  /p<J$. Then $1/r<p.J<J.J=J$, a contradiction. Hence
$J<1/\left(  rp\right)  $ and $rp<I$, by the definition of $I$.
\end{proof}

\begin{lemma}
\label{Lemma max ideal 5}$I$ is an ideal of $J$.
\end{lemma}

\begin{proof}
It is enough to prove that $Ip\leq I$ for all precise $p$ such that $1<p<J$.
By Lemma \ref{Lemma max ideal 4}
\[
I.p=\sup \left \{  r|e\left(  r\right)  =0\wedge r<I\right \}  \cdot
p=\sup \left \{  pr|e\left(  r\right)  =0\wedge r<I\right \}  \leq I.
\]

\end{proof}

\begin{lemma}
\label{Lemma max ideal 3}Let $K$ be an ideal of $J$. Then $K\leq I$ or $K=J$.
\end{lemma}

\begin{proof}
Assume that $K$ is an ideal of $J$ such that $I<K$. Then there exists a
precise $p$ such that $I<p<K$. Then $1/p<J$, so $1=p.1/p<K.1/p\leq K$. Suppose
$K<J$. Then there exists a precise $q$ such that $K<q<J$. Then $q=1.q<K.q\leq
K$, a contradiction. Hence $K=J$ or $K\leq I$.
\end{proof}

\bigskip

\begin{proof}
[Proof of Theorem \ref{Thm max ideal}]Lemma \ref{Lemma max ideal 5} states
that $I$ is an ideal of $J$. By Lemma \ref{Lemma max ideal 3}, every ideal of
$J$ which is different from $J$ is less than or equal to $I$. Hence $I$ is the
maximal ideal of $J$. The fact that $I\leq \oslash$ follows from Lemma
\ref{Lemma max ideal 1}.\ref{I<1/p}.
\end{proof}

We will now determine the value of the product of two magnitudes. Theorem
\ref{Product idempotents} below states that the product of idempotents is
equal to one of the factors and Proposition \ref{zerobar maximal ideal L}
deals with the special case of the product of $\oslash$ and $\pounds $.
Proposition \ref{Prop unicity idempotent} states that, for non-zero
magnitudes, the idempotent magnitude given by Axiom \ref{Axiom scale to ring}
is unique. Then the value of the product of magnitudes follows by applying
Axiom \ref{Axiom scale to ring} and Theorem \ref{Product idempotents} and
comes in the form of a linearization, otherwise said, the product of two
magnitudes is a multiple of one of them.

\begin{theorem}
\label{Product idempotents}Let $e$ and $f$ be idempotent magnitudes with
$e\leq f$. Let $I$ be the maximal ideal of $f$. Then $ef=e$ if $f<1$ or if
$1<f$ and $e\leq I$, otherwise $ef=f$.
\end{theorem}

\begin{proof}
If $e=f$ the property is obvious. If $e<f<1$, we have $ef=e$, for $e=ee\leq
ef\leq e\cdot1=e$. If $1<e<f$, we have $ef=f$, for $f=1\cdot f\leq ef\leq
ff=f$. Finally, assume that $e<1<f$. Assume that $e\leq I$. By the above we
have $eI=e$. Then by Axiom \ref{Axiom Maximal Ideal} one has $ef=eIf=eI=e$.
Assume that $I<e$. Note that $ef$ is an ideal of $f$, for $efq\leq eff=ef$ for
all precise $q<f$. Also $I<e=ee\leq ef\leq f$. Hence $ef=f$.
\end{proof}

Notice that by commutativity the product $ef$ is also defined if $f<e$.

\begin{proposition}
\label{zerobar maximal ideal L}The magnitude $\oslash$ is the maximal ideal of
$\pounds $. As a consequence $\oslash \pounds =\oslash$.
\end{proposition}

\begin{proof}
By Theorem \ref{Proposition Mult magnitudes}.\ref{L.L} the magnitude
$\pounds $ is idempotent. It follows from Lemma \ref{Lemma L<p}%
.\ref{zerobar sup} and Theorem \ref{Thm max ideal} that $\oslash$ is the
maximal ideal of $\pounds $. Then $\oslash \pounds =\oslash$ by Theorem
\ref{Product idempotents}.
\end{proof}

\begin{proposition}
\label{Prop unicity idempotent}Let $e$ be a nonzero magnitude and $I$ be an
idempotent magnitude such that $e=pI$ for some precise element $p$. Then $I$
is unique.
\end{proposition}

\begin{proof}
Let $e$ be a magnitude. Suppose that there exist idempotent magnitudes $I,J\in
S$ and precise elements $p$ and $q$ such that $e=pI$ and $e=qJ$. Then $pI=qJ$.
Because $e\neq0$ the elements $p$ and $q$ are non-zero. Then, noting that $I$
and $J$ are idempotent,%
\[
IJ=I\left(  \frac{p}{q}I\right)  =\frac{p}{q}I=J
\]
and
\[
IJ=\left(  \frac{q}{p}J\right)  J=\frac{q}{p}J=I.
\]
We conclude that $I=J$.
\end{proof}

\begin{theorem}
\label{pf or eq}Let $e$ and $f$ be magnitudes. Then the value of $ef$ is
uniquely determined. Moreover, there exists a positive precise $p$ such that
$ef=pf$ or a positive precise $q$ such that $ef=qe$.
\end{theorem}

\begin{proof}
The theorem is trivial if $e=0$ or $f=0$. Let $e$ and $f$ be non-zero
magnitudes. Let $I$ and $J$ be idempotent and $p$ and $q$ be precise such that
$e=pI$ and $f=qJ$; dealing with magnitudes they may be supposed positive. Then
$ef=pqIJ$. Now $IJ=I$ or $IJ=J$ by Theorem \ref{Product idempotents} and the
value of the product $ef$ is uniquely determined by Proposition
\ref{Prop unicity idempotent}. Hence $ef=pf$ or $ef=qe$.
\end{proof}

In the last part of this section we verify that the value of the product
obtained from Axiom \ref{Axiom Maximal Ideal}\ is consistent with the ordering
(Theorem \ref{Lemma consistency with order}) and the notion of supremum of
Axiom \ref{Axiom Dedekind completeness} (Theorem \ref{Thm consistency sup}).

\begin{lemma}
\label{Lemma max ideal order}Let $J>1$ be an idempotent magnitude and $I$ be
the maximal ideal of $J$. Let $p$ be a precise element.

\begin{enumerate}
\item \label{pJ<I}If $p<I$ then $pJ<I$.

\item \label{J<=pJ}If $I<p$ then $J\leq pJ$. Moreover, $pJ=J$ if and only if
$I<p<J$.

\item \label{J<pI}If $J<p$ then $J<pI$.

\item \label{pI<=I}If $p<J$ then $pI\leq I$. Moreover, $pI=I$ if and only if
$I<p<J$.

\item \label{No pI=J}There is no precise element $p$ such that $pI=J$.
\end{enumerate}
\end{lemma}

\begin{proof}
\ref{pJ<I}. Assume that $p<I$. Then $\sqrt{p}<I$ by Lemma
\ref{Lemma max ideal 1}.\ref{sqrt p<I}. Then $J<1/\sqrt{p}$. Hence
\[
pJ<p\frac{1}{\sqrt{p}}=\sqrt{p}<I.
\]

\ref{J<=pJ}. Assume that $I<p$. Then $I<p^{2}$ by idempotency. If $1\leq p$,
clearly $J\leq pJ$. If $p<1$, then $p<J$. Suppose that $pJ<J$. Because $pJ$ is
an ideal of $J$, one has $pJ\leq I<p^{2}$. Then $J<p$, a contradiction. Hence
$J\leq pJ$. Assume that $I<p<J$. Then $pJ\leq J^{2}=J$.\ Hence $pJ=J$. For
$J<p$ we have $J<p<pJ$. Hence $pJ=J$ if and only if $I<p<J$.

\ref{J<pI}. Suppose that $J<p$. Then $J<\sqrt{p}$ by idempotency. Then
$1/\sqrt{p}<I$, hence
\[
J<\sqrt{p}=p\frac{1}{\sqrt{p}}<pI.
\]

\ref{pI<=I}. Because $I$ is an ideal of $J$ we have $pI\leq I$. Assume that
$I<p<J$. Then $I=I^{2}\leq pI$. Hence $pI=I$. By Part \ref{pJ<I} one has
$pI\leq pJ<I$ for $p<I$. Hence $pI=I$ if and only if $I<p<J$.

\ref{No pI=J}. Directly from Part \ref{J<pI} and Part \ref{pI<=I}.
\end{proof}

\begin{theorem}
\label{Lemma consistency with order}Let $J>1$ be an idempotent magnitude and
$I$ be the maximal ideal of $J$. Let $p,q>0$ be precise.
\end{theorem}

\begin{enumerate}
\item \label{const 1}If $p<I<q$ then $pJ<IJ<qJ$.

\item \label{const 2}If $p<J<q$ then $Ip\leq IJ<Iq$.
\end{enumerate}

\begin{proof}
\ref{const 1}. By Lemma \ref{Lemma max ideal order}.\ref{pJ<I} and
\ref{Lemma max ideal order}.\ref{J<=pJ}
\[
pJ<I=IJ<J\leq qJ.
\]

\ref{const 2}. By Lemma \ref{Lemma max ideal order}.\ref{pI<=I} and
\ref{Lemma max ideal order}.\ref{J<pI}%
\[
Ip\leq I=IJ<J<Iq.
\]

\end{proof}

\begin{theorem}
\label{Thm consistency sup}Let $I,J$ be idempotent magnitudes such that $1<J$
and $I$ is the maximal ideal of $J$. Then
\end{theorem}

\begin{enumerate}
\item \label{const sup}$I=IJ=\sup \left \{  pJ\left \vert e(p)=0,\left \vert
p\right \vert <I\right.  \right \}  =\max \left \{  Iq\left \vert e(q)=0,\left \vert
q\right \vert <J\right.  \right \}  $.

\item \label{inc inf}$J=\inf \left \{  pI\left \vert e(p)=0,J<p\right.  \right \}
=\min \left \{  qJ\left \vert e(q)=0,I<q\right.  \right \}  $.
\end{enumerate}

\begin{proof}
\ref{const sup}. By Lemma \ref{Lemma max ideal order}.\ref{pJ<I}, if $p<I$
then $pJ<I$. Also $1<J$. Hence%
\[
I=\sup \left \{  p\left \vert e(p)=0,\left \vert p\right \vert <I\right.  \right \}
\leq \sup \left \{  pJ\left \vert e(p)=0,\left \vert p\right \vert <I\right.
\right \}  \leq I.
\]
Hence $IJ=\sup \left \{  pJ\left \vert e(p)=0,\left \vert p\right \vert <I\right.
\right \}  $. Also $Iq\leq IJ$ for all precise $q$ with $\left \vert
q\right \vert <J$. By Lemma \ref{Lemma max ideal order}.\ref{pI<=I} one has
$Iq=I=IJ$ for all precise $q$ with $I<\left \vert q\right \vert <J$. Hence
$IJ=\max \left \{  Iq\left \vert e(q)=0,\left \vert q\right \vert <J\right.
\right \}  $.

\ref{inc inf}. By Lemma \ref{Lemma max ideal order}.\ref{J<pI} we have
$J\leq \inf \left \{  pI|J<p\right \}  $. In order to show that also
$\inf \{pI|J<p\} \leq J,$ suppose towards a contradiction that $J<\inf \left \{
pI|J<p\right \}  $. Then there exists a precise element $q$ such that
$J<q<\inf \left \{  pI|J<p\right \}  $. Since $J<q$, one has $\inf \left \{
pI|J<p\right \}  \leq qI$. Then $q<qI,$ which implies that $1<I$, a
contradiction. One concludes that $\inf \left \{  pI|J<p\right \}  =J$. By Lemma
\ref{Lemma max ideal order}.\ref{J<pI}, $\inf \left \{  pI|J<p\right \}  $ is not
a minimum. By Proposition \ref{Zup E magnitude}, $\inf \left \{  pI|J<p\right \}
$ is an infimum.

By Lemma \ref{Lemma max ideal order}.\ref{J<=pJ} one has $J\leq pJ$ for $I<p$,
and in particular $pJ=J$ for all $p$ with $I<p<J$. Hence $J=\min \left \{
pJ|I<p\right \}  $.
\end{proof}

Theorem \ref{Thm consistency sup} also states that we may obtain $IJ=I$ by
approximation from below, but not by approximation from above. This shows that
completion arguments are not enough to determine the product of magnitudes.

\section{On consistency\label{section on consistency}}

In this section we show that Axioms \ref{assemblyassoc}%
-\ref{axiom archimedes2} are consistent by constructing a model extending a
particular nonstandard model of the real numbers. Indeed, we take a
sufficiently saturated\ nonstandard model $^{\ast}\mathbb{R}$ of the real
numbers which is elementary equivalent to $\mathbb{R}$. Within this model we
consider cosets with respect to convex subgroups which are definable by
$\Sigma_{1}$ or $\Pi_{1}$ formulas. The resulting structure will be called
$\mathcal{E}$. As we will see, all the axioms presented in Section
\ref{Section Axioms} are valid in $\mathcal{E}$.

In a previous article \cite{dinisberg 2011} we made an interpretation of most
of the algebraic axioms using the language of Nelson's Internal Set Theory
\cite{Nelsonist}. More precisely we used an adapted version, formulated by
Kanovei-Reeken \cite{kanoveireeken1}, which permits to include external sets.
As it turns out, in this approach the collection of magnitudes is a proper
class. To avoid foundational problems, which may appear when we apply, say,
asymptotics with parameters or defining subclasses, we will consider here a
semantic approach.

The axioms for a solid, i.e. Axioms \ref{assemblyassoc}%
-\ref{scheiding neutrices}, extend the axioms originally presented in
\cite{dinisberg 2011} and were shown to be consistent in \cite{dinisberg 2015
-1} by the construction of a direct model in the language of $ZFC$. This was
given in the form of a set of cosets of a non-Archimedean field. Allowing for
definable classes, it was shown in \cite{dinisberg 2011} that the external
numbers of \cite{koudjetithese} and \cite{koudjetivandenberg} satisfy the
axioms for addition and for multiplication, together with a modified form of
the distributivity axiom; this modified form was shown to be equivalent to
Axiom \ref{Axiom distributivity} in \cite{dinisberg 2015 -2}. The remaining
algebraic axioms deal with multiplication of magnitudes, and are in fact taken
from calculation rules of the external numbers observed in
\cite{koudjetithese} and \cite{koudjetivandenberg}.

\subsection{Construction of the solid $\mathcal{E}$%
\label{Subsection first order}}

Let $Z_{n}=\bigcup \nolimits_{k\leq n}P^{k}(\mathbb{R)}$ and $Z=\bigcup
\nolimits_{n\in \mathbb{N}}Z_{n}$ be the superstructure of Zakon-Robinson
\cite{Robinson Zakon} (see also \cite{Stroyan-Luxemburg} and \cite{Goldblatt}%
). Let $^{\ast}Z$ be an adequate ultralimit \cite{Nelsonist} of $Z$. If we
interpret the elements of $Z$ as standard, we will see that bounded versions
of Nelson's $IST$ axioms as well as his Reduction Algorithm hold in this
structure. In particular, the \emph{Saturation Principle} \cite[Thm.
5]{Nelsonsintax} holds. In the context of the superstructure this implies that
if $X\in Z$ and $s:X\rightarrow{}^{\ast}\mathbb{R}$, then $s$ has always an
internal extension $\widetilde{s}:{}^{\ast}X\rightarrow{}^{\ast}\mathbb{R}$.

\begin{definition}
\label{definition N}We denote by $\mathcal{N}$ the set consisting of $^{\ast
}\mathbb{R}$ and of all convex subgroups of $^{\ast}\mathbb{R}$ of the form $%
{\textstyle \bigcup_{_{x\in X}}}
\left[  -s_{x},s_{x}\right]  $ or $%
{\textstyle \bigcap_{_{x\in X}}}
\left[  -s_{x},s_{x}\right]  $, where $X\in Z$ and $s:X\rightarrow{}^{\ast
}\mathbb{R}$. We call an element of $\mathcal{N}$ a \emph{neutrix}.
\end{definition}

Without restriction of generality we may suppose that $X$ is ordered and $s$
is increasing in the case of unions and decreasing in the case of intersections.

\begin{definition}
\label{Definition Minkowski neutrices}Let $A,B$ be neutrices. With some abuse
of language we call the set $\{a\,{}^{\ast}+b:a\in A\wedge b\in B\}$ the
\emph{Minkowski sum} of $A$ and $B$ and the set $\{a\,{}^{\ast}\cdot b:a\in
A\wedge b\in B\}$ the \emph{Minkowski product }of $A$ and $B$.\emph{\ }Usually
we simply write $A+B$ instead of $A\,{}^{\ast}+B$ and $A\cdot B$ instead of
$A\,{}^{\ast}\cdot B$.
\end{definition}

\begin{definition}
\label{Definition E}We define $\mathcal{E}=\left \{  a+A|a\in{}^{\ast
}\mathbb{R}\wedge A\in{}\mathcal{N}\right \}  $. We call an element of
$\mathcal{E}$ an \emph{external number}.
\end{definition}

If $\alpha=a+A$ is an external number, it is tacitly understood that $a\in
{}^{\ast}\mathbb{R}$ and $A\in{}\mathcal{N}$. We write $N\left(
\alpha \right)  $ instead of $A$ and call it the \emph{neutrix part} of
$\alpha$. This functional notation will be justified below.

\begin{proposition}
Let $\alpha=a+A$ be an external number.

\begin{enumerate}
\item Let $y\in \alpha$. Then $\alpha=y+A$.

\item Let $\alpha=b+B$ with $b\in{}^{\ast}\mathbb{R}$ and $B\in{}\mathcal{N}$.
Then $A=B$.
\end{enumerate}
\end{proposition}

\begin{proof}
\begin{enumerate}
\item We have that $y-a \in \alpha- \alpha=A$. Then $y+A=y-a+a+A \subseteq a+
A+A =\alpha$. On the other hand, $\alpha=a+A=a-y+y+A \subseteq y+A+A=y+A$.
Hence $\alpha=y+A$.

\item We have $A=\alpha-\alpha=B$.
\end{enumerate}
\end{proof}

\begin{corollary}
The neutrix part is a well-defined function from $\mathcal{E}$ to
$\mathcal{N}$.
\end{corollary}

Obviously $\alpha=y+A$ for any element $y\in \alpha$. The neutrix part of a
given external number is unique and functional. Next definition extends the
Minkowski sum and product of Definition \ref{Definition Minkowski neutrices}
to external numbers. It is easy to see that the definition does not depend on
the choice of representatives.

\begin{definition}
Let $\alpha=a+A$ and $\beta=b+B$ be two external numbers, the\emph{\ sum} and
\emph{product} of $\alpha$ and $\beta$ are defined as follows
\begin{align*}
\alpha+\beta &  =a+b+A+B\\
\alpha \cdot \beta &  =ab+aB+bA+AB.
\end{align*}

\end{definition}

Let $A,B$ be neutrices and $\left(  s\,_{x}\right)  _{x\in X},\left(
t\,_{y}\right)  _{y\in Y}$ be families of elements of $^{\ast}\mathbb{R}$,
with $X,Y\in Z$. Because we are only considering convex subgroups of $^{\ast
}\mathbb{R}$ of the form $%
{\textstyle \bigcup_{_{x\in X}}}
[-s\,_{x},s\,_{x}]$\ or $%
{\textstyle \bigcap_{_{x\in X}}}
[-s\,_{x},s\,_{x}]$, we need to show that the sum and product operations do
not increase the complexity. With addition complexity does not increase
because one always has $A+B=A$ or $A+B=B$. We now consider multiplication.
There is clearly no increase in complexity if both $A$ and $B$ are unions, or
are intersections. Indeed
\[%
{\textstyle \bigcup_{_{x\in X}}}
\left[  -s\,_{x},s\,_{x}\right]  \cdot%
{\textstyle \bigcup_{_{y\in Y}}}
\left[  -t\,_{y},t\,_{y}\right]  =%
{\textstyle \bigcup_{_{(x,y)\in X\times Y}}}
\left[  -s\,_{x},s\,_{x}\right]  [-t\,_{y},t\,_{y}],
\]
and
\[%
{\textstyle \bigcap_{_{x\in X}}}
\left[  -s\,_{x},s\,_{x}\right]  \cdot%
{\textstyle \bigcap_{_{y\in Y}}}
\left[  -t\,_{y},t\,_{y}\right]  =%
{\textstyle \bigcap_{_{(x,y)\in X\times Y}}}
\left[  -s\,_{x},s\,_{x}\right]  [-t\,_{y},t\,_{y}].
\]
In the following proposition we show that there is also no increase in
complexity in the case where $A$ is of the form $%
{\textstyle \bigcup_{_{x\in X}}}
\left[  -s\,_{x},s\,_{x}\right]  $ and $B$ is of the form $%
{\textstyle \bigcap_{_{y\in Y}}}
\left[  -t\,_{y},t\,_{y}\right]  $.

\begin{proposition}
\label{Proposition complexity}Let $X,Y\in Z$. Let $A=%
{\textstyle \bigcup_{_{x\in X}}}
\left[  -s\,_{x},s\,_{x}\right]  $ and $B=%
{\textstyle \bigcap_{_{y\in Y}}}
\left[  -t\,_{y},t\,_{y}\right]  $, be neutrices where $\left(  s\,_{x}%
\right)  _{x\in X},\left(  t\,_{y}\right)  _{y\in Y}$ are families of elements
of $\,{}^{\ast}\mathbb{R}$. Then $AB=%
{\textstyle \bigcup_{_{w\in W}}}
\left[  -u\,_{w},u\,_{w}\right]  $ or $AB=%
{\textstyle \bigcap_{_{w\in W}}}
\left[  -u\,_{w},u\,_{w}\right]  $, where $\left(  u\,_{w}\right)  _{w\in W}$
is a family of elements of $\,{}^{\ast}\mathbb{R}$ with $W=X$ or $W=Y$.
\end{proposition}

\begin{proof}
To $AB$ we associate the halfline $C=\left]  -\infty,AB\right]  $. This
halfline is of the form $%
{\textstyle \bigcup_{_{x\in X}}}
\left]  -\infty,u\,_{x}\right]  $ where $u\,:X\rightarrow{}^{\ast}\mathbb{R}$
is internal or of the form $%
{\textstyle \bigcap_{_{y\in Y}}}
\left]  -\infty,u\,_{y}\right]  $ where $u\,:Y\rightarrow{}^{\ast}\mathbb{R}$
is internal (see \cite[Thm. 4.33]{vdbnaa}). The proposition is a direct
consequence of this fact.
\end{proof}

\begin{definition}
\label{order external numbers}Let $^{\ast}\leq$ be the order relation
on$\,{}^{\ast}\mathbb{R}$. Given $\alpha,\beta \in \mathcal{E}$, we write with
some abuse of language $\alpha \leq \beta$, if and only if
\begin{equation}
(\forall x\in \alpha)(\exists y\in \beta)(x\,{}^{\ast}\leq y). \label{def}%
\end{equation}

\end{definition}

Let $\alpha \in \mathcal{E}$. Let $Q_{\alpha}=\left \{  x\in \mathcal{E}\left \vert
x\leq \alpha \right.  \right \}  $. Then $\alpha \leq \beta$ if and only if
$Q_{\alpha}\subseteq Q_{\beta}$.

Note that if $\alpha \cap \beta=\emptyset$, formula (\ref{def}) is equivalent to
$(\forall x\in \alpha)(\forall y\in \beta)(x\,^{\ast}<y)$. Lemma
\ref{Tricotomia} shows that two external numbers are always either disjoint or
one contains the other (see also \cite[Prop. 3.2.15]{koudjetithese}).

\begin{lemma}
\label{Tricotomia}Let $\alpha$ and $\beta$ be two external numbers. Then%
\[
\alpha \cap \beta=\emptyset \vee \alpha \subseteq \beta \vee \beta \subseteq
\alpha \text{.}%
\]

\end{lemma}

\begin{proof}
Suppose that $\alpha \cap \beta \neq \emptyset$. Then there is $x\in{}^{\ast
}\mathbb{R}$ such that $x\in \alpha$ and $x\in \beta$. Then we may write
$\alpha=x+A$ and $\beta=x+B$. Hence $\beta \subseteq \alpha$ if $\max(A,B)=A$,
and $\alpha \subseteq \beta$ if $\max(A,B)=B$.
\end{proof}

\subsection{The solid $\mathcal{E}$ as a model for the
axioms\label{Subsection axiom schemes}}

In this section we show that the external numbers of the previous section are
a model for the\ axioms. We will work progressively, and start with the
algebraic axioms of a solid.

\begin{theorem}
\label{E, first-order}The structure $\left(  \mathcal{E},+,\cdot,\leq \right)
$ satisfies Axioms \ref{assemblyassoc}-\ref{scheiding neutrices}.
\end{theorem}

In order to prove the theorem, we verify first that the axioms for addition
and the axioms for multiplication are satisfied. For the order axioms we will
need to recapitulate in a modified way some results from \cite{koudjetithese}
and \cite{koudjetivandenberg}. Then we show that the axioms relating addition
and multiplication are satisfied and finally we show that the existence axioms
are verified.

\begin{proposition}
\label{E assemblies}The structure $\left(  \mathcal{E},+,\cdot,\leq \right)  $
satisfies Axioms \ref{assemblyassoc}-\ref{axiom u(xy)}.
\end{proposition}

\begin{proof}
The proof is essentially the same as the proof given in \cite[Thm.
4.10]{dinisberg 2011}.
\end{proof}

\begin{proposition}
\label{E order}The structure $\left(  \mathcal{E},+,\cdot,\leq \right)  $
satisfies Axioms \ref{(OA)reflex}-\ref{Axiom Amplification}.
\end{proposition}

Before proving the proposition, we recall a lemma from
\cite{koudjetivandenberg}.

\begin{lemma}
\label{Abeta<=Agamma}Let $A$ be a neutrix and let $\beta$ and $\gamma$ be
external numbers such that $\beta \leq \gamma$. Then $A\beta \subseteq A\gamma$.
\end{lemma}

\begin{proof}
Assume that $\beta \leq \gamma$. Let $x\in \beta$ and $a\in A$. There exists
$y\in \gamma$ such that $x\leq y$. Then $\left \vert a\right \vert x\leq
\left \vert a\right \vert y\in A\gamma$. Hence $A\beta \subseteq A\gamma$.
\end{proof}

\bigskip

\begin{proof}
[Proof of Proposition \ref{E order}]Working with halflines, it is immediate to
see that the order relation is reflexive, transitive, antisymmetric and total.
Then Axioms \ref{(OA)reflex}-\ref{(OA)total} are satisfied.

In order to show that Axiom \ref{(OA)compoper} is satisfied assume that
$\alpha \leq \beta$. Let $a\in \alpha$ and $c\in \gamma$. There exists $b\in \beta$
such that $a\leq b$. Hence $a+c\leq b+c\in \beta+\gamma$ and one concludes that
$\alpha+\gamma \leq \beta+\gamma$. As regards to Axiom \ref{Axiom e(x)maiory},
assume that $\alpha+N\left(  \beta \right)  =N\left(  \beta \right)  $, i.e.,
$a+A+B=B$. Then $a+A\subseteq B$. Hence $\alpha \leq N\left(  \beta \right)  $.
We now turn to Axiom \ref{compat mult}. Assume that $N\left(  \alpha \right)
\leq \alpha$ and $\beta \leq \gamma$. If $\alpha=A$, then $A\beta \subseteq
A\gamma$ by Lemma \ref{Abeta<=Agamma}, so $A\beta \leq A\gamma$. If $A<\alpha$
and $x\in \alpha$ then $0<x$. Let $y\in \beta$. Because $\beta \leq \gamma$ there
exists $z\in \gamma$ such that $y\leq z$. Then $xy\leq xz$. Hence $\alpha
\beta \leq \alpha \gamma$. Finally, to prove that Axiom \ref{Axiom Amplification}
holds suppose that $N\left(  \beta \right)  \leq \beta$ and $\beta \leq \gamma$.
Let $z\in A\beta$. We may assume that $z$ is positive. Then there exist
$a^{\prime}\in A,b^{\prime}\in \beta$ such that $z=a^{\prime}b^{\prime}$,
moreover $a^{\prime}$ may be supposed positive. Because $b^{\prime}\in \beta$
there is $c^{\prime}\in \gamma$ such that $b^{\prime}\leq c^{\prime}$. Then
$a^{\prime}b^{\prime}\leq a^{\prime}c^{\prime}\in A\gamma$. Hence
$A\beta \subseteq A\gamma$, which implies that $A\beta \leq A\gamma$.
\end{proof}

We turn now to the axioms which relate addition and multiplication. It was
shown in \cite{dinisberg 2011} that distributivity holds for external numbers
under certain conditions. In \cite{dinisberg 2015 -2} equivalence was shown
with Axiom \ref{Axiom distributivity}. Here we give a direct proof that Axiom
\ref{Axiom distributivity} holds in $\mathcal{E}$. We recall that for external
numbers, being convex sets, subdistributivity always holds in the sense of
inclusion, i.e.%
\begin{equation}
\alpha \left(  \beta+\gamma \right)  \subseteq \alpha \beta+\alpha \gamma.
\label{subdistributivity}%
\end{equation}

\begin{theorem}
\label{formula dist total n externos}Let $\alpha=a+A,\beta$ and $\gamma$ be
external numbers. Then%
\[
\alpha \beta+\alpha \gamma=\alpha \left(  \beta+\gamma \right)  +A\beta
+A\gamma \text{.}%
\]

\end{theorem}

\begin{proof}
It is easy to see \cite{dinisberg 2011} that distributivity holds in the case
that $\alpha=a$ is precise and that $\left(  a+A\right)  \left(  \beta
+\gamma \right)  =a\left(  \beta+\gamma \right)  +A\left(  \beta+\gamma \right)
$. Then%
\begin{align*}
\alpha \left(  \beta+\gamma \right)  +A\beta+A\gamma &  =\left(  a+A\right)
\left(  \beta+\gamma \right)  +A\beta+A\gamma \\
&  =a\left(  \beta+\gamma \right)  +A\left(  \beta+\gamma \right)
+A\beta+A\gamma \\
&  =a\beta+a\gamma+A\left(  \beta+\gamma \right)  +A\beta+A\gamma \text{.}%
\end{align*}
By formula (\ref{subdistributivity}) and because $A\beta$ and $A\gamma$ are
neutrices one has%
\[
\alpha \left(  \beta+\gamma \right)  +A\beta+A\gamma=a\beta+a\gamma
+A\beta+A\gamma \text{.}%
\]
Hence
\[
\alpha \left(  \beta+\gamma \right)  +A\beta+A\gamma=\alpha \beta+\alpha
\gamma \text{.}%
\]

\end{proof}

\begin{proposition}
\label{E, mixed}The structure $\left(  \mathcal{E},+,\cdot,\leq \right)  $
satisfies Axioms \ref{Axiom escala}-\ref{Axiom s(xy)=s(x)y}.
\end{proposition}

\begin{proof}
For Axioms \ref{Axiom e(xy)=e(x)y+e(y)x}, \ref{e(u(x))=e(x)d(x)} and
\ref{Axiom s(xy)=s(x)y} we refer to \cite[Prop. 4.15 and 4.17]{dinisberg
2011}. Axiom \ref{Axiom distributivity} holds by Theorem
\ref{formula dist total n externos}. We still must show that Axiom
\ref{Axiom escala} is satisfied. Let $A\in{}\mathcal{N}$ and $\beta=b+B$
$\in{}\mathcal{E}$. One has%
\[
A\left(  b+B\right)  =bA+AB=\max(bA,AB).
\]
Clearly $bA\in{}\mathcal{N}$, and $AB\in{}\mathcal{N}$ follows from
Proposition \ref{Proposition complexity}.
\end{proof}

We consider now the group of axioms on existence. We prove first the existence
of representatives of the special elements $m,u$ and $M$.

\begin{proposition}
\label{Proposition mim max}The structure $\left(  \mathcal{E},+,\cdot
,\leq \right)  $ satisfies Axioms \ref{Axiom neut min}-\ref{Axiom neut max}.
\end{proposition}

\begin{proof}
The proposition follows by putting $m=0$, $M={}^{\ast}\mathbb{R}$, and $u=1$ respectively.
\end{proof}

Axiom \ref{existencia neutrices} states the existence of magnitudes between
the smallest element $m$ and the largest element $M$. With Generalized
Dedekind Completeness we defined a largest magnitude $\oslash$ such that
$0<\oslash<1$ and a smallest magnitude $\pounds $ such that $1<\pounds <M$. We
will interpret $\oslash$ and $\pounds $ in the following way, where we
identify $\mathbb{N}$ with the standard integers of $^{\ast}\mathbb{R}$.

\begin{definition}
\label{Definition Labda and Theta}We define $\Lambda=%
{\textstyle \bigcup_{n\in \mathbb{N}}}
{}^{\ast}\left]  -n,n\right[  $ and $\Theta=%
{\textstyle \bigcap_{n\in \mathbb{N}}}
{}^{\ast}\left]  -\frac{1}{n},\frac{1}{n}\right[  $.
\end{definition}

\begin{theorem}
\label{Prop 0 dif L...}The external sets $\Theta$ and $\Lambda$ are neutrices.
One has $0<\Theta<1<\Lambda<{}^{\ast}\mathbb{R}$. The interpretation of
$\oslash$ is $\Theta$ and the interpretation of $\pounds $ is $\Lambda$.
\end{theorem}

\begin{proof}
Clearly $\Theta$ and $\Lambda$ are neutrices. Because $^{\ast}\mathbb{R}$ is a
superstructure of $\mathbb{R}$ there exists an infinitely large element $\nu$
in $^{\ast}\mathbb{R}$. Clearly $\nu \notin \Lambda$ hence $\Lambda \neq{}^{\ast
}\mathbb{R}$. Also $0\neq1/\nu \in \Theta$. Obviously $\Theta<1<\Lambda$,
because $1\in \Lambda$ and $1\notin \Theta$. Let $L,I$ be the interpretations of
$\pounds $ and $\oslash$ respectively. Then $L,I$ must be neutrices. There
does not exist a neutrix $A$ such that $1<A<\Lambda$, for there does not exist
a proper subset of $\mathbb{N}$ closed under addition. This implies also that
there does not exist a neutrix $B$ such that $\Theta<B<1$. Hence
$\Lambda \subseteq L$ and $I\subseteq \Theta$. Because $\Lambda+\Lambda=\Lambda$
and $1<\Lambda,$ by the definition of $\pounds $ one has $L\subseteq \Lambda$.
Also, because $\Theta+\Theta=\Theta$ and $\Theta<1,$ by the definition of
$\oslash$ one has $\Theta \subseteq I$. Hence $\Lambda=L$ and $\Theta=I$.
\end{proof}

\begin{proposition}
\label{E existence}The structure $\left(  \mathcal{E},+,\cdot,\leq \right)  $
satisfies Axioms \ref{existencia neutrices}-\ref{scheiding neutrices}.
\end{proposition}

\begin{proof}
By Theorem \ref{Prop 0 dif L...}, Axiom \ref{existencia neutrices} holds.
Axiom \ref{numexterno} is trivially satisfied. Finally we turn to Axiom
\ref{scheiding neutrices}. Let $A,B\in \mathcal{N}$ be such that $A\neq B$. We
may assume without loss of generality $A\varsubsetneq B$. Then there is a
nonstandard real number $b$ such that $b\in B$ and $b\notin A$. Furthermore,
$b$ may be supposed positive. We show that $A<b<B$. Indeed, because $B$ is a
group and $b$ is positive one has $b<2b\in B$. Hence $b<B$. Suppose that
$b\leq A$. Then there exists $a\in A$ such that $b\leq a$. Because $0\in A$
and $a\in A$, by convexity $b\in A$, a contradiction. Hence $A<b$.
\end{proof}

\bigskip

\begin{proof}
[Proof of Theorem \ref{E, first-order}]The theorem follows by combining
Proposition \ref{E assemblies}, Proposition \ref{E order}, Proposition
\ref{E, mixed}, Proposition \ref{Proposition mim max} and Proposition
\ref{E existence}.
\end{proof}

The set $\mathcal{E}$ is not characterized by Axioms \ref{assemblyassoc}%
-\ref{scheiding neutrices} for as showed in \cite{dinisberg 2015 -1} the set
of all cosets with respect to all convex subgroups for addition of a
non-Archimedean field is a model for these axioms. As we will see in the next
section all the algebraic axioms, i.e. Axioms \ref{assemblyassoc}%
-\ref{scheiding neutrices} together with the axioms \ref{Axiom Maximal Ideal}
and \ref{Axiom scale to ring}, are still not sufficient for such a characterization.

We will now prove that the Generalized Completeness Axiom
\ref{Axiom Dedekind completeness} holds in $\mathcal{E}$. We deal with this
axiom before the axioms on multiplication of magnitudes, because Generalized
Completeness is needed to prove the existence of the maximal ideals of Axiom
\ref{Axiom Maximal Ideal}.

In \cite[Thm. 4.34, Corollary 4.35]{vdbnaa} (see also \cite[p. 155]%
{Dienernsaip}) a normal form for convex subsets of real numbers is stated. In
the case of a (external) lower halfline this normal form indicates that its
upper boundary is well-defined, in the form of a unique external number. The
proof relies, in an essential way, on Nelson's Reduction Algorithm and on the
Saturation Principle.

Let $Z$ be the superstructure defined in the previous section. In order to
prove that the axiom on generalized completeness holds we interpret formulas
from the language $\{+,\cdot,\leq \}$ in the adequate ultralimit $^{\ast}Z$ and
show that a bounded version of the Reduction Algorithm as well as the
Saturation Principle hold in this structure.

\begin{definition}
Let $k$ be a natural number. Let $\Phi(x_{1},...,x_{k})$ be a formula of the
language $\{+,\cdot,\leq \}$ with free variables $x_{1},...,x_{k}$. The formula
$\Phi$ is called \emph{restricted} if each quantifier ranges over precise numbers.
\end{definition}

\begin{definition}
Let $k$ be a natural number. Let $\Phi(x_{1},...,x_{k})$ be a formula of $ZFC$
with free variables $x_{1},...,x_{k}$. The formula $\Phi$ is called
\emph{bounded} (relatively to $Z\cup{}^{\ast}Z$) if for each $i$ with $1\leq
i\leq k$ there exists an element $X_{i}\in Z\cup{}^{\ast}Z$ such that
$x_{i}\in X_{i}$ and each quantifier ranges over either an element of $Z $ or
an element of $^{\ast}Z$.
\end{definition}

\begin{definition}
Let $k$ be a natural number. A bounded formula $\Phi(x_{1},...,x_{k})$ is
called \emph{internal} (with some abuse of language) if all its quantifiers
range over elements of $^{\ast}Z$.
\end{definition}

Let $\Phi(x_{1},...,x_{k})$ be a restricted formula of the language
$\{+,\cdot,\leq \}$. We will interpret $\Phi$ by a formula $\bar{\Phi}$ in the
structure $\mathcal{E}$ by induction on the complexity of the formula, and
show that $\bar{\Phi}$ is bounded. Observe that a term $t(x_{1},...,x_{k})$ is
the result of a finite number of additions and multiplications of the
variables $x_{1},...,x_{k}$. Each variable $x_{i}$ with $1\leq i\leq k$ is
interpreted by an element, say, $\alpha_{i}$ of $\mathcal{E}$; in particular,
if $x_{i}$ is precise, then $\alpha_{i}\in{}^{\ast}\mathbb{R}$. Then the
interpretation $\overline{t}(\alpha_{1},...,\alpha_{k})$ of $t$ is the result
of a finite number of additions and multiplications of the elements
$\alpha_{1},...,\alpha_{k}$. An atomic formula is of the form $t(x_{1}%
,...,x_{k})\leq s(y_{1},...,y_{m})$, where $m$ is a natural number and $s$ is
a term with variables $(y_{1},...,y_{m})$. Then its interpretation is of the
form
\begin{equation}
\overline{t}(\alpha_{1},...,\alpha_{k})\leq \overline{s}(\beta_{1}%
,...,\beta_{m}), \label{term inequality}%
\end{equation}
with $\beta_{1},...,\beta_{m}\in \mathcal{E}$. It follows from Definition
\ref{definition N} that the $\alpha_{i}$ and $\beta_{j}$ are either unions or
intersections of families of intervals in $^{\ast}\mathbb{R}$ indexed by
elements of sets which are elements of $Z$. Hence the inequality
(\ref{term inequality}) is expressed by a bounded formula.

Clearly the negation of a bounded formula is a bounded formula, and the
conjunction of bounded formulas is a bounded formula. Since quantifiers in
restricted formulas of the language $\{+,\cdot,\leq \}$ range over precise
elements, quantifiers in their interpretations range over $^{\ast}\mathbb{R}$,
hence yield bounded formulas.

We conclude that the interpretation $\overline{\Phi}$ of $\Phi$ is bounded.

We show now that Nelson's Reduction Algorithm, properly adapted, transforms a
bounded formula into a bounded formula of the form $\forall x\in X\exists y\in
Y\,I\left(  x,y\right)  ,$ with $X,Y\in Z$ and $I\left(  x,y\right)  $
internal. Nelson's Reduction Algorithm uses three principles Transfer $\left(
T\right)  ,$ Idealization $\left(  I\right)  $ and modified Standardization
$\left(  S^{\prime}\right)  $.

By \cite{Nelsonist} the Transfer Axiom and the Idealization Axiom of $IST$,
when relativized to $^{\ast}Z$, hold in $^{\ast}Z$ indeed. In our context the
modified Standardization Axiom takes the following form. Let $\Phi(x,y)$ be a
bounded formula, this means that all quantifiers and parameters range over
some $Z_{n}$. Let $m,n\in \mathbb{N}$ and $X,Y$ such that $X\subseteq Z_{m}$
and $Y\subseteq Z_{n}$. Assume that $\forall x\in Z_{m}\exists y\in Z_{n}\,
\Phi(x,y)$. Then there must exist a function $\tilde{y}\in Z$ such that
$\forall x\in X\, \Phi(x,\tilde{y}\left(  x\right)  )$. This is true because
if $\Phi$ is a formula of $ZFC$, by the Axiom of Choice there exists
$\tilde{y}:X\mathbb{\rightarrow}Y$ such that $\forall x\in X\, \Phi
(x,\tilde{y}(x))$. Clearly $\tilde{y}\in Z$. So $\left(  S^{\prime}\right)  $
also holds in $Z$.

All three principles transform bounded formulas into bounded formulas. By the
reasoning in the paragraph above this is clearly true for $\left(  S^{\prime
}\right)  $. We verify the property also for $\left(  T\right)  $ and $\left(
I\right)  $. Let in the formulas below $\Phi$ always be a bounded formula.
Then $\left(  T\right)  $ becomes%
\[
\forall y\in Y(\forall x\in X\, \Phi \left(  x,y\right)  \leftrightarrow \forall
x\in{}^{\ast}X\, \Phi \left(  x,y\right)  ),
\]
where $X,Y\in Z$ and $\Phi$ internal. Also $\left(  I\right)  $ becomes%
\begin{multline*}
\forall w\in{}^{\ast}W(\forall v\in P_{fin}\left(  X\right)  \exists y\in
{}^{\ast}Y\forall x\in v\, \Phi \left(  x,y,w\right) \\
\leftrightarrow \exists y\in{}^{\ast}Y\forall x\in{}^{\ast}X\, \Phi \left(
x,y,w\right)  ),
\end{multline*}
where $X,Y,W\in Z,$ $\Phi$ internal and $P_{fin}\left(  X\right)  $ is the set
of all finite subsets of $X$. Note that $P_{fin}\left(  X\right)  \in Z$.

So we have the following theorem.

\begin{theorem}
Every bounded formula $\Phi$ is equivalent to a bounded formula of the form
$\forall x\in X\exists y\in Y\,I\left(  x,y\right)  ,$ with $X,Y\in Z$ and
$I\left(  x,y\right)  $ internal.
\end{theorem}

Because the Saturation Principle is true in bounded $IST$ it also holds in
$^{\ast}Z$. \ We may now apply \cite[Thm. 4.34, Cor. 4.35]{vdbnaa} to show
that Generalized Dedekind Completeness holds.

\begin{theorem}
\label{Dedekind holds in E}Axiom \ref{Axiom Dedekind completeness} holds in
$\mathcal{E}$.
\end{theorem}

\begin{proof}
The interpretation $\overline{A}\left(  a\right)  $ of $A\left(  x\right)  $,
with $a\in{}^{\ast}\mathbb{R}$, is a bounded formula. Hence one can apply the
Reduction Algorithm to $\overline{A}\left(  a\right)  $ to obtain an
equivalent formula of the form $\forall u\in U\exists v\in V\,B(u,v,a)$, with
$U,V\in Z$ and $B$ internal. Since $\overline{A}\left(  a\right)  $ defines a
lower halfline, by \cite[Thm 4.33]{vdbnaa} this formula can be reduced to a
formula of the form $\exists y\in Y\,C(y,a)$ or $\forall y\in Y\,C(y,a)$, with
$Y\in Z$ and $C$ internal. Then the result follows by \cite[Thm. 4.34, Cor.
4.35]{vdbnaa}.
\end{proof}

We show now that the two axioms on multiplication of magnitudes hold in the
model $\left(  \mathcal{E},+,\cdot,\leq \right)  $. We recall that magnitudes
are interpreted by convex groups. Next proposition states that the
interpretation of an idempotent magnitude larger than $1$ is a ring with
unity, the interpretation of an ideal in a solid is an ideal in the algebraic
sense and that under these interpretations the product of an idempotent
magnitude and its maximal ideal is equal to this maximal ideal.

\begin{proposition}
\label{Lemma interpretation ideal}Let $S$ be a complete solid. Let $J\in S$ be
an idempotent magnitude such that $1<J$. Let $I$ be an ideal of $J$. In the
model $\left(  \mathcal{E},+,\cdot,\leq \right)  $, the interpretation $\bar
{J}$ of $J$ is a ring and the interpretation $\bar{I}$ of $I$ is an ideal of
the ring $\bar{J}$. Moreover, if $I$ is maximal, then $\bar{I}=\left \{
1/x|x\in{}^{\ast}\mathbb{R},J<\left \vert x\right \vert \right \}  \cup \left \{
0\right \}  $ is maximal and $\bar{I}\bar{J}=\overline{I}$.
\end{proposition}

\begin{proof}
The interpretation $\bar{J}$ of $J$ in $\mathcal{E}$ is an idempotent neutrix,
which is clearly a ring. An ideal in the sense of Definition
\ref{Definition ideal} is a magnitude, so $\bar{I}$ is a neutrix. Because for
all $y<J$ one has $yI\leq I$, by the Minkowski definition of the product
$xz\in \bar{I}$ for all $x\in \bar{I}$ and $z\in \bar{J}$. This means that
$\bar{I}$ is an ideal of $\bar{J}$ in the sense of rings.

Assume now that $I$ is maximal, then $I=\sup \left \{  1/\omega|\omega \text{
precise, }J<\left \vert \omega \right \vert \right \}  $ by Theorem
\ref{Thm max ideal}. Let $K\equiv \left \{  1/x|x\in{}^{\ast}\mathbb{R},\bar
{J}<\left \vert x\right \vert \right \}  \cup \left \{  0\right \}  $. We show that
$\bar{I}=K$. Suppose that there exists $y\in \bar{I}\backslash K$. Then
$\left \vert 1/y\right \vert <\bar{J}$. Hence there exists $u<I$ such that
$1/u<J$, in contradiction with the definition of $I$. Hence $\bar{I}\subseteq
K$. Suppose that there exists $z\in K\backslash \bar{I}$. Then $1/z<\bar{J}$,
i.e. $1/z\in \bar{J}$, in contradiction with the definition of $K$. Hence
$K\subseteq \bar{I}$ and we conclude that $\bar{I}=K$. Suppose the ring
$\bar{J}$ has an ideal $L$ with $\bar{I}\subset L\subset \bar{J}$. Let $x\in
L\backslash \bar{I},x<1$ be positive. Because $\bar{I}=K$ we may find $y\in
\bar{J}\backslash L$ such that $1/x<y$. Then $y^{2}\in \bar{J}$. But
$xy^{2}\notin L$, since $y<xy^{2}$. So we have a contradiction. As a
consequence $\bar{I}$ is the maximal ideal of the ring $\bar{J}$.

As observed above, $yz\in \bar{I}$ for all $y\in \bar{J}$ and $z\in \bar{I}$.
Again by the Minkowski definition of the product, it holds that $\bar{I}%
\bar{J}\subseteq \overline{I}$. Clearly $\bar{I}\subseteq \bar{I}\cdot
1\subseteq \overline{I}\overline{J}$. Hence $\bar{I}\bar{J}=\overline{I}$.
\end{proof}

\begin{corollary}
\label{E max ideal}The structure $\left(  \mathcal{E},+,\cdot,\leq \right)  $
satisfies Axiom \ref{Axiom Maximal Ideal}.
\end{corollary}

In the syntactical setting of (bounded) $IST$ Axiom \ref{Axiom scale to ring}
is verified using an argument based on the logarithm and the exponential
function \cite[Thm. 7.4.4]{koudjetivandenberg}. It can be adapted without
difficulty to our semantic setting.

\begin{theorem}
\label{E scale to ring}The structure $\left(  \mathcal{E},+,\cdot,\leq \right)
$ satisfies Axiom \ref{Axiom scale to ring}.
\end{theorem}

Finally we prove that the axioms on the existence and behavior of natural
numbers hold in $\mathcal{E}$.

\begin{theorem}
\label{Natural numbers hold in E}Let $N$ be interpreted by$\ ^{\ast}%
\mathbb{N}$, the set of non-negative nonstandard integers of $^{\ast
}\mathbb{R}$. Then Axioms \ref{Axiom natural numbers}-\ref{axiom archimedes2}
hold in $\mathcal{E}$.
\end{theorem}

\begin{proof}
We interpret the symbol $+$ by the addition in $\mathcal{E}$, the symbol
$\cdot$ by the multiplication in $\mathcal{E}$ and the symbol $\leq$ by the
order relation in $\mathcal{E}$. This corresponds with the addition $^{\ast}+$
, the multiplication $^{\ast}\cdot$ and the order relation $^{\ast}\leq$ in
$^{\ast}\mathbb{N}$. Then Axiom \ref{Axiom natural numbers} holds because
$^{\ast}\mathbb{N}$ does not contain negative numbers, $^{\ast}0\in{}^{\ast
}\mathbb{N}$ and whenever $n\in{}^{\ast}\mathbb{N}$, $n+1\in{}^{\ast
}\mathbb{N}$, but $y\notin{}^{\ast}\mathbb{N}$ for any $y\in$ $^{\ast
}\mathbb{R}$ with $n<y<n+1$. Axiom \ref{Axiom *s-induction} states that
induction is valid for each formula $A$ with the symbols $0$, $1$, $+$ and
$\cdot$, and precise variables which have the property $N$. Then its
interpretation $^{\ast}A$ is a formula with the symbols $^{\ast}0$, $^{\ast}%
1$, $^{\ast}+$ and $^{\ast}\cdot$, with parameters interpreted by elements of
$^{\ast}\mathbb{N}$, and quantifications ranging over $^{\ast}\mathbb{N}$.
Because $^{\ast}\mathbb{N}$ is a model of Peano Arithmetic, Axiom
\ref{Axiom *s-induction} holds in $\mathcal{E}$ indeed. As regards to Axiom
\ref{axiom archimedes2}, it follows from Lemma \ref{precise separation} that
it is enough to show that the axiom holds for precise elements. Let $x,y\in$
$^{\ast}\mathbb{R}$ be such that $0<x<y$. By construction, for all $a\in
{}^{\ast}\mathbb{R}$ there exists $n\in{}^{\ast}\mathbb{N}$ such that $a<n$.
In particular there exists $m\in^{\ast}\mathbb{N}$ such that $y/x<m$. Hence
$y<mx$, so Axiom \ref{axiom archimedes2} holds in $\mathcal{E}$.
\end{proof}

\begin{theorem}
The structure $\left(  \mathcal{E},+,\cdot,\leq \right)  $ satisfies Axioms
\ref{assemblyassoc}-\ref{axiom archimedes2}.
\end{theorem}

\begin{proof}
Directly from Theorem \ref{E, first-order}, Theorem \ref{Dedekind holds in E},
Corollary \ref{E max ideal}, Theorem \ref{E scale to ring} and Theorem
\ref{Natural numbers hold in E}.
\end{proof}

\begin{corollary}
Axioms \ref{assemblyassoc}-\ref{axiom archimedes2} are consistent with $ZFC$.
\end{corollary}

\section{Complete arithmetical
solids\label{Section Characterization properties}}

\begin{definition}
A model $E$ for Axioms \ref{assemblyassoc}-\ref{axiom archimedes2} will be
called a \emph{complete arithmetical solid}. The set of magnitudes of $E$ will
be denoted by $\mathcal{N}_{E}$. The set of precise numbers of $E$ will be
denoted by $\mathcal{P}_{E}$. If there is no ambiguity we drop the subscript
$E$ and write simply $\mathcal{N}$, respectively $\mathcal{P}$.
\end{definition}

In the previous section we showed that the structure $\mathcal{E}$ given by
Definition \ref{Definition E} is in fact a complete arithmetical solid. This
structure was based on the superstructure $Z$ over the set of real numbers
$\mathbb{R}$ and on the nonstandard model $^{\ast}\mathbb{R}$ of an ultralimit
$^{\ast}Z$ of $Z$. Its set of magnitudes was given in Definition
\ref{definition N} and its set of precise numbers was $^{\ast}\mathbb{R}$.
Even if a set of magnitudes is specified in the above way it is to be expected
that the set of precise numbers is not uniquely determined. Indeed, we would
then have a first-order characterization of a set of real numbers, for the
axioms of Section \ref{Section Axioms} are stated within first-order logic.
However, we will show that the set of non-precise numbers is completely
determined. This will be a consequence of Theorem \ref{Lemma rational} which
states that if a set of magnitudes is specified in a complete arithmetical
solid, the set of non-precise numbers is completely determined as sums of
nonstandard rationals and a magnitude.

For the set of precise numbers we obtain lower and upper bounds. Indeed,
Theorem \ref{Thm charac2} states that the set of precise numbers is
necessarily a nonstandard ordered field situated between the nonstandard
rationals and the nonstandard reals; the field is Archimedean for the
corresponding set of nonstandard natural numbers. This is to be compared with
the well-known theorem saying that an Archimedean ordered field lies between
the rationals and the reals. The "standard" structure related to this field is
also situated between rationals and reals.

It will be shown that the precise elements of a complete arithmetical solid
satisfy the axioms of $ZFL$ \cite{Lutz}. The theory $ZFL$ is basically a
calculatory nonstandard axiomatics in which the Leibniz rules hold. In
\cite{Callot} it is shown that $ZFL$ is sufficient to develop a nonstandard
Calculus in terms of $S$-continuity, $S$-differentiability and $S$%
-integrability. The axiomatics $ZFL$ is weaker than Nelson's arithmetical
axiomatics of Radically Elementary Probability Theory \cite{REPT} due to the
lack of the axiom scheme of External Induction. This axiomatics is here called
$REPT$. In $REPT$ the axiom scheme of External Induction holds for formulas on
the language $\left \{  \st,\in \right \}  $. Nelson shows that it is possible to
do advanced stochastics in $REPT$. We show that in a complete arithmetical
solid External Induction holds for formulas in the language $\left \{
\st,+,\cdot \right \}  $.

In Subsection \ref{Solid rho E} we show that the algebraic axioms alone are
not sufficient to characterize the external numbers by exhibiting a proper
substructure $^{\rho}\mathcal{E}$ of $\mathcal{E}$ satisfying all the
algebraic axioms. This justifies the introduction of the arithmetical axioms.

In Subsection \ref{Subsection induction} we show that every complete
arithmetical solid contains a copy of a nonstandard model of Peano arithmetic.
As a consequence, in our framework we have a copy of the nonstandard
rationals. By analogy to the construction of the reals via Dedekind cuts we
show in Subsection \ref{Subsection precise and non-precise} that the precise
numbers of a complete arithmetical solid are situated between the nonstandard
rationals and the nonstandard reals. The proof that a complete arithmetical
solid has two built-in models of the rational numbers uses a notion of
standard part, here called shadow and is based on the well-known construction
of the standard reals as the quotient of the rationals by the infinitesimals.
As a consequence, a complete arithmetical solid $E$ can only be constructed in
a nonstandard setting. In Subsection \ref{Subsection CAS and NSA} we compare
our axiomatics with the nonstandard axiomatics $ZFL$ and $REPT$.

The results in this section suggest that our axiomatic approach gives rise to
an alternative way to build nonstandard real numbers, sharing the algebraic
spirit of Benci and Di Nasso \cite{Benci di Nasso}.

It is useful to identify magnitudes $f$ of a solid $S$ with the set $P_{f}$ of
its precise elements, i.e.
\begin{equation}
P_{f}\equiv \left \{  x\in S|e\left(  x\right)  =0\wedge \left \vert x\right \vert
<f\right \}  . \label{identification}%
\end{equation}
With some abuse of language the sets $P_{f}$ will also be called magnitudes.

\begin{proposition}
\label{Proposistion identification}Let $S$ be a solid. Let $X$ be the set of
all magnitudes in $S$ and $P\left(  S\right)  $ be the set of all subsets of
$S$. Let $\phi:X\rightarrow P\left(  S\right)  $ be the map defined by
$\phi \left(  f\right)  =P_{f}$, where $P_{f}$ is given by
(\ref{identification}). Then $P_{f}$ is a convex subgroup of $S$ for the
addition and order relation of $S$. The map $\phi$ is $1-1$.
\end{proposition}

\begin{proof}
It is clear that $P_{f}$ is a convex subgroup of $S$ for the addition and
order relation of $S$. To prove that $\phi$ is $1-1$, assume that $f,g\in X$
with $f<g$. Then there exists a precise element $p$ such that $f<p<g$. Then
$p\in Pg$ and $p\notin P_{f}$. Hence $P_{f}\subset P_{g}$.
\end{proof}

With some abuse of language, we identify $\pounds $ with the set
$P_{\pounds }$ and $\oslash$ with the set $P_{\oslash}$. Elements of
$\pounds $ are called \emph{limited} and elements of $\oslash$ are called
\emph{infinitesimal}.

\subsection{The solid $^{\rho}\mathcal{E}$\label{Solid rho E}}

We show that the algebraic axioms alone are not sufficient for a
characterization of the external numbers. We do this by exhibiting a proper
substructure $^{\rho}\mathcal{E}$ of $\mathcal{E}$ that also satisfies all the
algebraic axioms. In this way we also obtain that the symbols $\oslash$,
respectively $\pounds $ as defined in (\ref{Zerobar El}) may have an
interpretation different from the infinitesimals, respectively the limited numbers.

Indeed, let $\rho \in{}^{\ast}\mathbb{R},\rho>0$ be infinitely large. We define
$G=%
{\textstyle \bigcup \limits_{n\in \mathbb{N}}}
\left[  -\rho^{n},\rho^{n}\right]  $ and $H=%
{\textstyle \bigcap \limits_{n\in \mathbb{N}}}
\left[  -\left(  1/\rho \right)  ^{n},\left(  1/\rho \right)  ^{n}\right]  $.
Clearly $G$ and $H$ are idempotent. The field $G/H\equiv{}^{\rho}\mathbb{R}$
was studied by Lightstone and Robinson in \cite{Lightstone-Robinson}.

\begin{definition}
\label{Definition ro}We define $^{\rho}\mathcal{I}$ as the set of all convex
sets $I\subseteq{}^{\ast}\mathbb{R}$ of the form $%
{\textstyle \bigcup \limits_{n\in \mathbb{N}}}
\left[  -p_{n},p_{n}\right]  $ or $%
{\textstyle \bigcap \limits_{n\in \mathbb{N}}}
\left[  -\left(  1/p_{n}\right)  ,\left(  1/p_{n}\right)  \right]  $, with
$p_{n}>0$, and $p_{n+1}/p_{n}$ increasing such that $p_{0}=1,p_{1}\geq \rho$
and $p_{n+1}/p_{n}\geq p_{n}^{1/n}$ for all $n\in%
\mathbb{N}
,n>0$. We let $^{\rho}\mathcal{N}$ be the set of all neutrices of the form
$qI$ where $q\in{}^{\ast}\mathbb{R}$ and $I\in \mathcal{{}}^{\rho}\mathcal{I}$,
and $^{\rho}\mathcal{E}$ as the set of elements of $\mathcal{E}$ of the form
$r+L$ where $r\in{}^{\ast}\mathbb{R}$ and $L\in{}^{\rho}\mathcal{N}\cup \{0\}
\cup \{^{\ast}\mathbb{R}\}$.
\end{definition}

\begin{proposition}
\label{ro I closed}The set $^{\rho}\mathcal{I}$ consists of idempotent
neutrices, with minimal element greater than $1$ equal to $G$ and maximal
element less than $1$ equal to $H$. Moreover, $^{\rho}\mathcal{I}$ is closed
under addition and multiplication and satisfies Axiom
\ref{Axiom Maximal Ideal}.
\end{proposition}

\begin{proof}
Let $J$ be of the form $%
{\textstyle \bigcup \limits_{n\in \mathbb{N}}}
\left[  -p_{n},p_{n}\right]  $ and $I$ of the form\ $%
{\textstyle \bigcap \limits_{n\in \mathbb{N}}}
[-\left(  1/p_{n}\right)  ,$\linebreak$\left(  1/p_{n}\right)  ]$, with
$(p_{n})_{n\in \mathbb{N}}$ as given by Definition \ref{Definition ro}. Then
$p_{n+1}/p_{n}\geq \rho$ for all $n\in%
\mathbb{N}
,n>0$, which implies that both $I$ and $J$ are neutrices. Also $p_{n}^{2}\leq
p_{2n}$ for all $n\in%
\mathbb{N}
,n>0$, which implies that both $I$ and $J$ are idempotent. The set $^{\rho
}\mathcal{I}$ is closed under addition, because its elements are neutrices.
Then $^{\rho}\mathcal{I}$ is closed under multiplication by Theorem
\ref{Product idempotents}. Because $p_{n}\geq \rho^{n}$ for all $n\in%
\mathbb{N}
$, one has $J\supseteq G$ and $I\subseteq H$. So $G$ is the minimal element of
$^{\rho}\mathcal{I}$ greater than $1$, and $H$ is the maximal element of
$^{\rho}\mathcal{I}$. less than $1$. Also $I=\left \{  1/x|x\in{}^{\ast
}\mathbb{R},J<\left \vert x\right \vert \right \}  \cup \left \{  0\right \}  $.
Then Proposition \ref{Lemma interpretation ideal} implies that $I$ is the
maximal ideal of the ring $J$ and that $IJ=I$. Hence $^{\rho}\mathcal{I}$
satisfies Axiom \ref{Axiom Maximal Ideal}.
\end{proof}

\begin{proposition}
\label{R is solid}The set $^{\rho}\mathcal{E}$ satisfies Axioms
\ref{assemblyassoc}-\ref{Axiom scale to ring}.
\end{proposition}

\begin{proof}
Axiom \ref{Axiom Maximal Ideal} holds by Proposition \ref{ro I closed}. Axiom
\ref{Axiom scale to ring} and the existence axioms hold by construction. The
remaining axioms hold because $^{\rho}\mathcal{E}$ is a substructure of
$\mathcal{E}$.
\end{proof}

By Proposition \ref{ro I closed}, the symbol $\oslash$ can be interpreted by
$H$, and the symbol $\pounds $ can be interpreted by $G$. Indeed, in $^{\rho
}\mathcal{E}$ the set of neutrices less than $1$ has a weak supremum, in fact
a maximum, in the form of $H$, while $G$ is the weak infimum (minimum) of the
set of neutrices larger than $1$.

To show whether the Generalized Dedekind completeness axiom holds in $^{\rho
}\mathcal{E}$, one should establish that definable lower halflines in $^{\rho
}\mathcal{E}$ have a weak supremum. This would require a deeper study of
polynomials of external numbers, which falls outside of the scope of this
article; observe that, due to the fact that distributivity does not hold in
full generality, the product of polynomials does not need to be a polynomial.
However, the introduction of natural numbers via the arithmetical axioms
permits to distinguish between $\mathcal{E}$ and $^{\rho}\mathcal{E}$. Indeed,
as will be shown below, the set $\mathcal{E}$ contains a copy of $^{\ast
}\mathbb{N}$, and induction holds in $^{\ast}\mathbb{N\cap}\pounds $, but not
in $^{\ast}\mathbb{N\cap}G$. For example, in $^{\ast}\mathbb{N\cap}G$ the
domain of function $x\mapsto2^{x}$ is closed under the successor function, but
this function is not total.

\subsection{On induction in complete arithmetical
solids\label{Subsection induction}}

The interpretation of $N$ in a complete arithmetical solid will be denoted by
$^{\ast}\mathcal{K}$. We show that $^{\ast}\mathcal{K}$\ satisfies the axioms
of Peano Arithmetic. Let $A$ be a formula of the language of Peano Arithmetic
which we denote by $\mathbf{L}=\{ \mathbf{0,1,}\boldsymbol{+}\mathbf{,\bullet
,s\}}$. We may extend the language $L=\left \{  +,\cdot,\leq \right \}  $ to a
language $L^{\prime}$ which includes the symbols $m,u,Sc$, corresponding
respectively to $\mathbf{0,1,s}$. Indeed, let $m$ be the neutral element for
addition as in Axiom \ref{Axiom neut min}. It is easy to see that it is unique
and therefore definable. The same is true for the neutral element for
multiplication $u$ of Axiom \ref{Axiom neut mult}. Putting $Sc(x)=x+u$, we
obtain a definable successor function (functional relation) of one variable
$Sc$. In this way, the formula $A$ has a $1-1$ correspondence with a formula
$B^{\prime}$ in the extended language $L^{\prime}$ which may be seen as an
abbreviation of a formula $C^{\prime}$ of the original language $L$. We let
$B$ be the relativization of $B^{\prime}$ to $N$, and $C$ be the
relativization of $C^{\prime}$ to $N$. Then within a complete arithmetical
solid $E$ the interpretations of $A,B$ and $C$ are all the same.

\begin{theorem}
\label{*K Peano}Let $E$ be a complete arithmetical solid. Then $^{\ast
}\mathcal{K}$ satisfies the Peano Axioms as formulated in the language
$\mathbf{L}$.
\end{theorem}

\begin{proof}
Observe that all elements of $^{\ast}\mathcal{K}$ are precise because the
predicate $N$ only applies to precise variables. With some abuse of language
let $0$ be the neutral element of $E$, $1$ be the unity of $E$, $+$ the
addition on $E$ and $\cdot$ the multiplication on $E$. We may interpret $Sc$
by the function $\sigma:{}^{\ast}\mathcal{K\rightarrow{}}^{\ast}\mathcal{K}$
given by $\sigma \left(  k\right)  =k+1$. As observed above $0$ is the
interpretation of $\mathbf{0}$, $1$ is the interpretation of $\mathbf{1}$, $+$
is the interpretation of $\boldsymbol{+}$, $\cdot$ is the interpretation of
$\mathbf{\bullet}$ and $\sigma$ is the interpretation of $\mathbf{s}$.
By\ Axiom \ref{Axiom natural numbers}, $0\in{}^{\ast}\mathcal{K}$ and $1\in
{}^{\ast}\mathcal{K}$. Also by Axiom \ref{Axiom natural numbers} the function
$\sigma$ is well-defined because if $k\in{}^{\ast}\mathcal{K}$ then $k+1\in
{}^{\ast}\mathcal{K}$, and all elements of $^{\ast}\mathcal{K}$ are
non-negative. Because $1$ is precise, from $\sigma \left(  k\right)
=\sigma \left(  k^{\prime}\right)  $ we derive that $k=k^{\prime}$ for all
$k,k^{\prime}\in{}^{\ast}\mathcal{K}$, hence $\sigma$ is $1-1$. By Axiom
\ref{Axiom neut min} one has $k+0=k$ for all $k\in{}^{\ast}\mathcal{K}$. By
associativity it holds that $k+\sigma \left(  k^{\prime}\right)  =k+\left(
k^{\prime}+1\right)  =\left(  k+k^{\prime}\right)  +1=\sigma \left(
k+k^{\prime}\right)  $ for all $k,k^{\prime}\in{}^{\ast}\mathcal{K}$. By
\cite[Prop. 3.5]{dinisberg 2015 -2} one has that $k.0=0$ for all $k\in E$. It
follows from the fact that distributivity holds for precise elements that
$k.\sigma \left(  k^{\prime}\right)  =k\left(  k^{\prime}+1\right)
=kk^{\prime}+k$, for all $k,k^{\prime}\in{}^{\ast}\mathcal{K}$. Let $A\left(
x\right)  $ be a property of the language of Peano Arithmetic allowing for a
free variable $x$ such that $A\left(  0\right)  $ holds and for all $x$ if
$A\left(  x\right)  $ then $A\left(  \mathbf{s}\left(  x\right)  \right)  $.
As argued before $A$ corresponds to a formula $B$ of the language $L$,
relativized to $N$, which has the same interpretation $I$ in $E$. Then $I$ is
a subset of ${}^{\ast}\mathcal{K}$ such that $0\in I$ and whenever $k\in I$
one has $k+1\in I$. Now $\forall x(N(x)\rightarrow B(x))$ by Axiom
\ref{Axiom *s-induction}. Because the interpretation of $N$ is $^{\ast
}\mathcal{K}$ it follows that $I={}^{\ast}\mathcal{K}$. Hence $A\left(
x\right)  $ holds for all $x\in{}^{\ast}\mathcal{K}$, i.e. induction holds for
the formula $A$. We conclude that ${}^{\ast}\mathcal{K}$ is a model for Peano Arithmetic.
\end{proof}

By Theorem \ref{*K Peano}, induction in $^{\ast}\mathcal{K}$ holds for
formulas of $L^{\prime}$ with only precise variables, all relativized to $N$.
The set of limited elements of $^{\ast}\mathcal{K}$ will be denoted by
$\mathcal{K}$. We show below that $\mathcal{K}$ is also a model of Peano
Arithmetic, with induction being valid for all formulas $A\left(  x\right)  $
of $L^{\prime}$ allowing for a free precise variable $x$ and quantifications
only over precise variables, possibly with (non precise) parameters. The two
types of induction may be compared with Internal Induction and External
Induction of $IST$, the first valid for internal formulas, i.e. formulas in
the language $\{ \in \}$, and the second valid also for external formulas, i.e.
formulas in the language $\{ \in,\st \}$.

\begin{theorem}
\label{K Peano}Let $E$ be a complete arithmetical solid. Then $\mathcal{K}$
satisfies the Peano Axioms with induction over formulas $A\left(  x\right)  $
allowing for a free precise variable $x$ and quantifications only over precise
variables expressed with the symbols $+,\cdot$ possibly with (non precise)
parameters. In fact $^{\ast}\mathcal{K}$ is an end extension of $\mathcal{K}$.
\end{theorem}

\begin{proof}
By construction $\pounds $ contains $0$ and $1$. Because it is an idempotent
magnitude it is closed under addition and multiplication. The successor
function may be interpreted by the function $s:{}\mathcal{K\rightarrow{}K}$
given by $s\left(  k\right)  =k+1$. Then the Peano Axioms, except induction,
are proved along the lines of Theorem \ref{*K Peano}. Let $A\left(  x\right)
$ be a property in the language $L^{\prime}$ allowing for a free precise
variable $x$ and quantifications only over precise variables, possibly with
(non precise) parameters, such that $A\left(  0\right)  $ holds, and for all
$x,$ if $A\left(  x\right)  $ then $A\left(  x+1\right)  $. Then $A$ is
interpreted by a set $\bar{A}$ with $0\in \bar{A}$ and whenever $k\in \bar{A}$
one has $k+1\in \bar{A}$. We prove that $\mathcal{K\subseteq}\bar{A}$, i.e.
induction over $\mathcal{K}$ holds for the formula $A$. Let $B=\left \{
x|e\left(  x\right)  =0\wedge \exists a\in \bar{A}\left(  0\leq x\leq a\right)
\right \}  $. Let $\gamma=c+C=\operatorname*{zup}B$. Then $0\leq \gamma$ and we
may assume that $0\leq c$. Because $\bar{A}$ is closed under addition by $1$
it is impossible that $C\subseteq \oslash$, hence $\pounds \subseteq C$. Assume
$\gamma$ is a supremum. Then $c+C\subseteq B$. Hence $\mathcal{K={}}^{\ast
}\mathcal{K}\cap \pounds \subseteq{}^{\ast}\mathcal{K}\cap \lbrack
0,c+\pounds )\subseteq{}^{\ast}\mathcal{K}\cap \lbrack0,c+C)=\bar{A}$. Assume
$\gamma$ is not a supremum. Then $B=\left[  0,c+C[\right[  $ with $0<c+C$. We
have $c/2<c+C$ by Lemma \ref{precise separation}. Suppose that $c/2\in
\pounds $. Then $c\in \pounds $. Hence $B\subset \pounds $. We conclude that
$C\subseteq \oslash$, a contradiction. Hence $\mathcal{K={}}^{\ast}%
\mathcal{K}\cap \pounds \subseteq{}^{\ast}\mathcal{K}\cap \left[  0,c/2\right]
\subseteq{}^{\ast}\mathcal{K}\cap B=\bar{A}$. It follows that induction over
$\mathcal{K}$ holds for the formula $A$. Hence $\mathcal{K}$ is a model for
Peano Arithmetic. By construction $^{\ast}\mathcal{K}$ is an end extension of
$\mathcal{K}$.
\end{proof}

Given a complete arithmetical solid $E$, we denote by $^{\ast}\mathcal{Q}$ the
set of rational numbers constructed in the usual way from $^{\ast}\mathcal{K}$
and by $\mathcal{Q\subset{}}^{\ast}\mathcal{Q}$ the set of rational numbers
constructed from $\mathcal{K}$.

\subsection{Precise and non-precise elements of a complete arithmetical
solid\label{Subsection precise and non-precise}}

We start by showing that the non-precise elements of a complete arithmetical
solid are characterized by sums of rationals and magnitudes.

\begin{theorem}
\label{Lemma rational}Let $E$ be a complete arithmetical solid. Let $\tilde
{E}$ be the set of non-precise elements of $E$. Then $\tilde{E}$ $=\left \{
q+N|q\in{}^{\ast}\mathcal{Q}\wedge N\in \mathcal{N}\right \}  $.
\end{theorem}

\begin{proof}
Let $\xi \in \tilde{E}$. Then there exists $p\in \mathcal{P}$ and $N\in
\mathcal{N}$ such that $\xi=p+N$. Let $0<b<N$. Then ${}^{\ast}\mathcal{Q\cap
}\pounds \cap \left[  p,p+b\right]  \neq \emptyset$. Let $q\in{}^{\ast
}\mathcal{Q\cap}\pounds \cap \left[  x,x+b\right]  $. Then $q+N=p+N$.
\end{proof}

By the previous theorem, the set of magnitudes of a complete arithmetical
solid $E$ determines the set of non-precise elements. However, the nonstandard
rationals $^{\ast}\mathcal{Q}$ give only a lower bound for the precise
elements. We will show below that an upper bound is given by nonstandard
reals. Also, Theorem \ref{Lemma rational} allows to define standard precise
numbers. To see this we need to introduce some definitions and notation.

\begin{definition}
\label{Definition K-notions}Let $x\in E$ be precise and limited. Then
$x+\oslash$ is called the $\mathcal{K}$-\emph{shadow }of $x$. The
$\mathcal{K}$-\emph{shadow} of a subset $D$ of precise elements of $E$ is the
set of $\mathcal{K}$-shadows of all limited elements of $D$. We denote the
$\mathcal{K}$-shadow of $\mathcal{Q}$ by $\mathfrak{Q}$, the $\mathcal{K}%
$-shadow of $^{\ast}\mathcal{Q}$ by $\mathfrak{F}$ and the $\mathcal{K}%
$-shadow of $\mathcal{P}$ by $\mathfrak{P}$.
\end{definition}

\begin{definition}
Let $K$ be a model of Peano Arithmetic. An ordered field $F$ is $K$%
-\emph{Archimedean} if for all $a,b\in F,a,b>0$ there is $k\in K$ such that
$ka>b$. A $K$-\emph{real field} is an ordered field which is $K$-Archimedean
and such that every precise lower halfline has a least upper bound.
\end{definition}

If $\mathcal{K}=\mathbb{N}$ or $K=\mathbb{N}$, we may suppress the prefix in
the above notions, for they correspond to common notions. Observe that a
$\mathbb{N}$-real field is isomorphic to $\mathbb{R}$ and that the
$\mathbb{N}$-shadow of a limited real number is in $1-1$ correspondence with
the usual shadow $^{\circ}x$ of $x$, more commonly called standard part
\cite{Dienernsaip}.

The proof that $\mathcal{P}$ is contained in a copy of the nonstandard reals
uses the construction based on lower halflines given by Theorem
\ref{*R ordered field}. As regards the $K$-shadow of $\mathcal{P}$, we have
$\mathfrak{P=F}$ by Theorem \ref{Lemma rational}. Also, $\mathfrak{Q}%
\subseteq \mathfrak{P}$. For $\mathcal{K}=\mathbb{N}$ we have $\mathcal{Q=}%
\mathbb{Q}$, and $\mathfrak{P}$ itself is a real field. Indeed $\mathfrak{P}$
is isomorphic to the quotient of $^{\ast}\mathbb{Q}$ by the infinitesimals,
which is isomorphic to $\mathbb{R}$. In general, the set $\mathfrak{F}$ turns
out to be a $\mathcal{K}$-Archimedean ordered field, see Theorem
\ref{R ordered field} below. It is not necessarily a $\mathcal{K}$-real field,
for it may not be Dedekind complete, since in principle the Dedekind property
holds only for definable cuts. Then by Theorem \ref{*R ordered field} it can
be extended to a $\mathcal{K}$-real field.

\begin{definition}
\label{Definition *R}Let $K$ be a model of Peano Arithmetic and $F$ be a
$K$-\linebreak Archimedean ordered field. Similarly to Definition
\ref{Definition halfline}, a lower halfline $A$ is \emph{precise} if there is
no positive $d$ such that $a+d\in A$ for all $a\in A$. We define $R_{F}$ as
the set of precise lower halflines of elements of $F$ without maximal element.
For $f\in{}F$ we define $H\left(  f\right)  =\left \{  x\in{}F|x<f\right \}  $.
We put $^{\ast}\mathcal{R=}R_{^{\ast}\mathcal{Q}}$ and $\mathfrak{R}%
\mathcal{=}R_{\mathfrak{F}}$.
\end{definition}

Clearly a precise lower halfline $A$ is strictly contained in $F$ and for all
$k\in{}K$ there exist $a\in{}A$ and $c\in{}F\backslash{}A$ such that $c-a<1/k$.

The proof of the construction of a $K$-real field from a $K$-Archimedean
ordered field follows roughly the lines of the usual construction of the reals
as Dedekind cuts of the rationals.

\begin{theorem}
\label{*R ordered field}Let $K$ be a model of Peano Arithmetic and $F$ be a
$K$-\linebreak Archimedean ordered field. Then $R_{F}$ is a $K$-real field.
The mapping $H$ is an isomorphism of the field $F$ onto a subfield $H\left(
F\right)  \subseteq$ $R_{F}$.
\end{theorem}

\begin{proof}
As for the definition of addition, multiplication and order, the verification
of their first-order properties and the fact that $R_{F}$ contains a copy of
$F$ via the mapping $H$, one may follow the lines of a textbook proof of the
construction of the real numbers from the rationals using cuts, see for
instance \cite{Rudin}. By construction $R_{F}$ is Archimedean for $K$. We show
that the least upper bound property is satisfied for precise halflines. Let
$A\subset{}R_{F}$ be a precise lower halfline. Let
\[
b=\cup_{a\in A}a.
\]
Let $x\in b$ and $y\in{}F$ such that $y<x$. Then there exists $a\in A$ such
that $x\in a$. Because $a$ is a lower halfline of $F$ one has $y\in a$. Hence
$y\in b$. We conclude that $b$ is lower halfline of $F$. Also, since $a$ does
not have a maximal element, there exists $z\in a$ such that $x<z$. Then $z\in
b$ with $x<z$, and we see that $b$ does not have a maximal element. We show
that $b$ is precise. If not, there exists a positive $d$ $\in{}F$ such that
$z+d\in b$ whenever $z\in b$. Because $A$ is precise there exist $a\in A$ and
$c\notin A$ such that $c-a=d$. There exist elements $x,x^{\prime}\in{}F$ such
that $x^{\prime}-x<d$ $\ $with $x\in a$ and $x^{\prime}\notin a$. Also there
exist elements $y,y^{\prime}\in{}F$ such that $y^{\prime}-y<d$ $\ $with $y\in
c$ and $y^{\prime}\notin c$. Note that $y^{\prime}\notin b$ because
$y^{\prime}\notin c$ and $c\notin A$.\ On the other hand $y-x^{\prime}<d=c-a$
and
\[
y^{\prime}<x^{\prime}+y^{\prime}-y+d<x^{\prime}+2d<x+3d.
\]
Now$\ x\in b$, so $x+3d\in b$. Then $y^{\prime}\in b$ because $b$ is a lower
halfline, a contradiction. Hence $b$ is precise and we conclude that $b\in{}F$.

Finally we show that $b$ is the least upper bound of $A$. Let $a\in A$. Then
$a\leq b$ because $x\in b$ whenever $x\in F$ and $x\in a$. Suppose that
$b^{\prime}<b$. Then there exists $x\in b$ such that $x\notin b^{\prime}$.
Then $x\in a$ for some $a\in A$. Then $b^{\prime}<a$, hence $b^{\prime}$ is
not an upper bound of $A$.
\end{proof}

\begin{theorem}
Let $K$ be a model of Peano Arithmetic and $F$ be a $K$-\linebreak Archimedean
ordered field. Let $Q$ be the set of rational numbers corresponding to $K$.
Then $Q$ may be embedded in $F$ and here is a isomorphism between $R_{Q}$ and
$R_{F}$.
\end{theorem}

\begin{proof}
It is obvious that $Q$ may be embedded in $F$. To see that $R_{Q}$ and $R_{F}$
are isomorphic, define $\phi:R_{F}\rightarrow R_{Q}$ by $\phi(A)=A\cap Q$. To
see that $\phi$ is $1-1$, let $A,B\in R_{F}$ with $A\neq B$; we may suppose
that $A\subset B$. Then there exists $y\in B$ such that $x<y$ for all $x\in
A$. Because $B$ does not have a maximal element there exists $z\in B$ with
$y<z$, and because $F$ is $K$-Archimedean there exists $q\in Q$ such that
$y<q<z$. Then $q\in \phi(B)\backslash \phi(A)$. Hence $\phi(A)\neq \phi(B)$. To
see that $\phi$ is onto, let $B\in R_{Q}$. Put $A=\{x\in F\left \vert \exists
q\in B,x<q\right.  \}$. Then $A$ is a lower halfline of $F$ without a maximal
element, otherwise there would exist $q\in B\ $with $x<q$; then $q\in F$,
because $B$ does not have a maximal element. Clearly $A\cap Q=B$. Hence
$\phi(A)=B$. It is also obvious that $\phi$ is respects the algebraic
operations and the order. We conclude that there is a isomorphism between
$R_{Q}$ and $R_{F}$.
\end{proof}

As a consequence we obtain that $\mathcal{P}$ is contained in the $^{\ast
}\mathcal{K}$-real field $R_{\mathcal{P}}$, and can be embedded in $R_{^{\ast
}Q}$.

\begin{corollary}
\label{Thm charac1}Let $E$ be a complete arithmetical solid. Then
$R_{\mathcal{P}}$ is a $^{\ast}\mathcal{K}$-real field and the mapping $H$ is
an isomorphism of the field $\mathcal{P}$ onto a subfield $H\left(
\mathcal{P}\right)  \subseteq$ $^{\ast}\mathcal{R}$.
\end{corollary}

We will now prove that $\mathfrak{P}$ is a $\mathcal{K}$-Archimedean ordered
field situated between $\mathfrak{Q}$ and a $\mathcal{K}$-real field
$\mathfrak{R}$. This will be a consequence of next theorem which states that
$\mathfrak{F}$ is a $\mathcal{K}$-Archimedean ordered field.

\begin{theorem}
\label{R ordered field}Let $E$ be a complete arithmetical solid. The set
$\mathfrak{F}$ is a $\mathcal{K}$-Archimedean ordered field.
\end{theorem}

\begin{proof}
Let $x=p+\oslash,y=q+\oslash \in \mathfrak{F}$. Then $\left \vert p\right \vert
,\left \vert q\right \vert \in{}^{\ast}\mathcal{Q\cap}\pounds $. Hence
$x+y=p+q+\oslash \in \mathfrak{F}$ and $xy=\left(  p+\oslash \right)  \left(
q+\oslash \right)  =pq+p\oslash+q\oslash+\oslash \oslash=pq+\oslash
\in \mathfrak{F}$, by Lemma \ref{Lemma max ideal order}.\ref{pI<=I} and Theorem
\ref{Proposition Mult magnitudes}.\ref{zerobar.zerobar}. Because
$\mathfrak{F}$ is a substructure of $E$ both addition and multiplication are
commutative and associative in $\mathfrak{F}$. Then $\mathfrak{F}$ is an
abelian group for addition with neutral element $\oslash$ and inverse
$-p+\oslash$ because $p+\oslash+\left(  -p+\oslash \right)  =\oslash$.

By Lemma \ref{Lemma max ideal order}.\ref{pI<=I} one has $p\oslash \leq \oslash
$, so $(1/p)\oslash \leq \oslash$ by Lemma \ref{Lemma L<p}.\ref{1/p<L}. Then
$\left(  1+\oslash \right)  \left(  p+\oslash \right)  =p+\oslash+p\oslash
=p+\oslash$ and, using Theorem \ref{Proposition Mult magnitudes}%
.\ref{zerobar.zerobar}
\[
\left(  p+\oslash \right)  \left(  \frac{1}{p}+\oslash \right)  =1+p\oslash
+\frac{1}{p}\oslash+\oslash \oslash=1+\oslash \text{.}%
\]

Then $\mathfrak{F}\backslash \left \{  \oslash \right \}  $ is also an abelian
group for multiplication with neutral element $1+\oslash$ and inverse
$1/p+\oslash$ for the element $p+\oslash$. To prove distributivity, let
$x=p+\oslash,y=q+\oslash,z=r+\oslash \in \mathfrak{F}$. By Axiom
\ref{Axiom distributivity} we have $xy+xz=x\left(  y+z\right)  +e\left(
x\right)  y+e\left(  x\right)  z$. Now $e\left(  x\right)  y+e\left(
x\right)  z=\oslash \left(  q+\oslash \right)  +\oslash \left(  r+\oslash \right)
=\oslash$. Also%
\begin{align*}
x\left(  y+z\right)   &  =\left(  p+\oslash \right)  \left(  q+\oslash
+r+\oslash \right)  =\left(  p+\oslash \right)  \left(  q+r+\oslash \right) \\
&  =p\left(  q+r\right)  +p\oslash+\left(  q+r\right)  \oslash+\oslash
\oslash \\
&  =p\left(  q+r\right)  +\oslash.
\end{align*}
Then%
\[
xy+xz=x\left(  y+z\right)  +\oslash=x\left(  y+z\right)  .
\]
Hence distributivity holds, so $\mathfrak{F}$ is a field. Because
$\mathfrak{F}$ is a substructure of $E$ the order axioms are valid and we
conclude that $\mathfrak{F}$ is indeed an ordered field. By construction
$\mathfrak{F}$ is Archimedean for $\mathcal{K}$.
\end{proof}

Knowing that $\mathfrak{P=F}$ is $\mathcal{K}$-Archimedean ordered field, we
obtain a corollary to Theorem \ref{*R ordered field}.

\begin{corollary}
\label{Theorem standard}Let $E$ be a complete arithmetical solid. Then
$\mathfrak{R}$ is a $\mathcal{K}$-real field and the mapping $H$ is an
isomorphism of the field $\mathfrak{P}$ onto a subfield $H\left(
\mathfrak{P}\right)  \subseteq$ $\mathfrak{R}$.
\end{corollary}

Within a complete arithmetical solid we may now characterize the set of
precise numbers and its shadow by lower and upper bounds as follows.

\begin{theorem}
\label{Thm charac2}Let $E$ be a complete arithmetical solid. Then up to
identifications, $^{\ast}\mathcal{Q\subset P}\subseteq{}^{\ast}\mathcal{R}$
and $\mathfrak{Q}\subset \mathfrak{P=F}\subseteq \mathfrak{R}$.
\end{theorem}

In the natural case where $\mathcal{K=}\mathbb{N}$ we have some
simplifications. The set $^{\ast}\mathcal{R}$ becomes a nonstandard model
$^{\ast}\mathbb{R}$ of the reals and the set $^{\ast}\mathcal{Q}$ a
nonstandard model $^{\ast}\mathbb{Q}$ of the rationals. The field
$\mathfrak{F}$ can be identified with the quotient of the external set of
limited rationals $^{\ast}\mathbb{Q}$ of a nonstandard model $^{\ast
}\mathbb{R}$ of the reals by the infinitesimal rationals. So $\mathfrak{F}$
itself is already Dedekind complete and is isomorphic to $\mathbb{R}$. With
these identifications the set of precise elements is a proper extension of
$^{\ast}\mathbb{Q}$ and a cofinal subfield of $^{\ast}\mathbb{R}$. As a
consequence, up to identifications the lower and upper bounds of Theorem
\ref{Thm charac2} take the form $^{\ast}\mathbb{Q}\mathcal{\subset P}%
\subseteq{}^{\ast}\mathbb{R}$ and $\mathfrak{Q}\subset \mathfrak{P=}\mathbb{R}$.

\subsection{Complete arithmetical solids and nonstandard
analysis\label{Subsection CAS and NSA}}

In this final subsection we investigate the relation between the standard
structure and the nonstandard structure of a complete arithmetical solid. We
show that the precise numbers satisfy the axiomatics $ZFL$ and that the
nonstandard natural numbers $^{\ast}\mathcal{K}$ satisfy the axiomatics
$REPT$, with external induction restricted to the language $\left \{
+,\cdot \right \}  $. We recall that the language of $REPT$ is $\{ \st,\in \}$ and
its axioms are:

\begin{enumerate}
\item $\st(0)$;

\item $\forall n \in \mathbb{N} (\st(n) \to \st(n+1))$;

\item $\exists \omega \in \mathbb{N}(\lnot \st(\omega))$;

\item $(\Phi(0)\wedge \forallst n\in \mathbb{N}(\Phi(n)\rightarrow
\Phi(n+1)))\rightarrow \forallst n\, \Phi(n)$.
\end{enumerate}

In the last axiom $\Phi$ is an arbitrary formula, internal or external, and
$\forallst n\in \mathbb{N}\, \Phi(n)$ is an abbreviation of $\forall
n\in \mathbb{N}(\st(n)\rightarrow \Phi(n))$.

\begin{theorem}
Let $E$ be a complete arithmetical solid. Then $\mathcal{P\cap}\pounds $
satisfies the Leibniz rules. Moreover, if we interpret the limited elements as
elements of $\mathcal{P\cap}\pounds $ then $\mathcal{P}$ is a model of $ZFL$.
\end{theorem}

\begin{proof}
The result follows from the fact that $\pounds $ is an idempotent neutrix.
\end{proof}

\begin{theorem}
\label{Thm REPT}Let $E$ be a complete arithmetical solid. If we interpret the
standard numbers by elements of $\mathcal{K}$. Then $^{\ast}\mathcal{K}$ is a
model of $REPT$ with external induction restricted to the language $\left \{
+,\cdot \right \}  $.
\end{theorem}

\begin{proof}
We interpret $\st(n)$ by $n\in \mathcal{K}$. Then the result is a consequence
of Theorem \ref{K Peano}.
\end{proof}

Observe that in the special case where $\mathcal{K}=\mathbb{N}$ we even have
external induction in the language $\{ \st,\in \}$ and then $^{\ast}\mathcal{K}$
is a model of $REPT$.


\begin{thebibliography}{99}                                                                                               %


\bibitem {Benci di Nasso}V. Benci, M. Di Nasso, \textit{Alpha-theory: An
elementary axiomatics for nonstandard analysis}, Expositiones Mathematicae 21,
4 (2003) 355--386.

\bibitem {vdbnaa}I. P. van den Berg, \textit{Nonstandard Asymptotic Analysis,
}Springer Lecture Notes in Mathematics 1249 (1987).

\bibitem {neutrixdecompositie}I. P. van den Berg, \textit{A decomposition
theorem for neutrices}, Annals of Pure and Applied Logic 161, no. 7 (2010) 851-865.

\bibitem {Wallet}I. P. van den Berg, \textit{External borders and strongly
open sets}, in: Des Nombres et des Mondes, \'{E}. Benoit and J.-P. Furter
(eds.), \'{E}ditions Hermann, Paris (2012) 69-86.

\bibitem {debruijn}N. G. de Bruijn, \textit{Asymptotic Analysis}, North
Holland (1961).

\bibitem {Callot}J. L. Callot, \textit{Trois le\c{c}ons d'Analyse
Infinit\'{e}simale}, in J. M. Salanskis (ed.), Le labyrinthe du continu,
Springer France, Paris (1992) 369--381.

\bibitem {Van der Corput}J. G. van der Corput, \textit{Introduction to the
neutrix calculus}, J. Analyse Math. 7 (1959/1960) 281--399.

\bibitem {Dienernsaip}F. and M. Diener (eds.), \textit{Nonstandard Analysis in
Practice}, Springer Universitext (1995).

\bibitem {dinis 2017}B. Dinis, \textit{Old and new approaches to the Sorites
paradox} (submitted). Available from ArXiv: https://arxiv.org/abs/1704.00450.

\bibitem {dinisberg 2011}B. Dinis, I. P. van den Berg, \textit{Algebraic
properties of external numbers}, J. Logic and Analysis 3:9 (2011) 1--30.

\bibitem {dinisberg 2015 -1}B. Dinis, I. P. van den Berg, \textit{On the
quotient class of non-archimedean fields }(submitted). Available from ArXiv: http://arxiv.org/abs/1510.08714.

\bibitem {dinisberg 2015 -2}B. Dinis, I. P. van den Berg,
\textit{Characterization of distributivity in a solid (submitted)}. Available
from ArXiv: http://arxiv.org/abs/1510.08722.

\bibitem {Goldblatt}R. Goldblatt, \textit{Lectures on the hyperreals. An
introduction to nonstandard analysis}, Graduate Texts in Mathematics, 188.
Springer-Verlag, New York (1998).

\bibitem {Gonshor}H. Gonshor, \textit{Remarks on the Dedekind completion of a
nonstandard model of the reals}, Pacific J. Math. Volume 118, Number 1 (1985) 117-132.

\bibitem {Justino}J. Justino, I. P. van den Berg, \textit{Cramer's Rule
applied to flexible systems of linear equations}, Electronic Journal of Linear
Algebra 24 (2012) 126-152.

\bibitem {kanoveireeken0}V. Kanovei, M. Reeken, \textit{Mathematics in a
nonstandard world I}, Mathematica Japonica 45, no. 2 (1997) 369-408.

\bibitem {kanoveireeken1}V. Kanovei, M. Reeken, \textit{Nonstandard Analysis,
axiomatically}, Springer Monographs in Mathematics (2004).

\bibitem {KeislerSchmerl}H. J. Keisler, J. H. Schmerl, \textit{Making the
hyperreal line both saturated and complete}, J. Symbolic Logic 56 (1991) 1016-1025.

\bibitem {koudjetithese}F. Koudjeti, \textit{Elements of External Calculus
with an application to Mathematical Finance}, Ph.D. thesis, Labyrinth
publications, Capelle a/d IJssel, The Netherlands (1995).

\bibitem {koudjetivandenberg}F. Koudjeti, I.P. van den Berg,
\textit{Neutrices, external numbers and external calculus}, in: Nonstandard
Analysis in Practice, F. and M. Diener (eds.), Springer Universitext (1995) 145-170.

\bibitem {Lightstone-Robinson}A. H. Lightstone and A. Robinson,
\textit{Non-archimedean fields and asymptotic expansions}, Bulletin of the
American Mathematical Society 83 (1977), no. 2, 231--235.

\bibitem {Lutz}R. Lutz, \textit{R\^{e}veries infinit\'{e}simales}, La Gazette
des math\'{e}maticiens 34 (1987).

\bibitem {Nelsonist}E. Nelson, \textit{Internal Set Theory, an axiomatic
approach to nonstandard analysis},\ Bull. Am. Math. Soc., 83:6 (1977) 1165-1198.

\bibitem {Nelsonsintax}E. Nelson, \textit{The syntax of nonstandard analysis,
}Annals of Pure and Applied Logic, 38, (1988) 123-134.

\bibitem {REPT}E. Nelson, \textit{Radically elementary probability theory},
Annals of Mathematical Studies, vol. 117, Princeton University Press,
Princeton, N. J., 1987.

\bibitem {Robinson Zakon}A. Robinson, E. Zakon, \textit{A Set-Theoretical
Characterization of Enlargements}, in Applications of Model Theory to Algebra,
Analysis and Probability, W.A.J. Luxemburg (ed.), Holt, Rinehart and Winston,
New York (1969) 109-122.

\bibitem {Rudin}W. Rudin, \textit{Principles of Mathematical Analysis}, 3rd
ed. McGraw-Hill (1976).

\bibitem {Stroyan-Luxemburg}K. D. Stroyan, W. A. J. Luxemburg,
\textit{Introduction to the theory of infinitesimals}, Pure and Applied
Mathematics, No. 72. Academic Press New York-London (1976).

\bibitem {Wattenberg}F. Wattenberg, [0, $\infty$]\textit{-valued, translation
invariant measures on }$\mathit{%
\mathbb{N}
}$\textit{\ and the Dedekind completion of *$\mathbb{R}$}, Pacific J. Math.,
90 (1980) 223-247.
\end{thebibliography}
\end{document}